\documentclass[onecolumn, 12 pt]{IEEEtran}  

\IEEEoverridecommandlockouts

\usepackage{graphicx}
\usepackage{epsfig} 
\usepackage{times} 
\usepackage{amsmath}
\usepackage{amssymb}  
\usepackage{color}
\usepackage{amsfonts}
\usepackage{subfigure}
\usepackage{multirow}
\usepackage{multicol}
\usepackage{siunitx}
\usepackage{overpic}
\usepackage{mathrsfs}  

\usepackage{booktabs}
\usepackage{threeparttable}

\usepackage{epstopdf}
\newtheorem{lemma}{Lemma}

\newtheorem{theorem}{Theorem}

\newtheorem{assumption}{Assumption}
\newtheorem{corollary}{Corollary}
\usepackage{xspace}

\begin{document}

\title{\LARGE \bf Predictor-Feedback Stabilization of \\ Multi-Input Nonlinear Systems}

\author{Nikolaos Bekiaris-Liberis and Miroslav Krstic

\thanks{N. Bekiaris-Liberis is with the Department of Production Engineering \& Management, Technical University of Crete, Chania, 73100, Greece. Email address: \texttt{nikos.bekiaris@gmail.com}. 

M. Krstic is with the Department of Mechanical \& Aerospace Engineering, University of California, San Diego, La Jolla, CA, 92093, USA. Email address: \texttt{krstic@ucsd.edu}.} }

\maketitle
\thispagestyle{empty}

\begin{abstract}
\baselineskip=1.7\normalbaselineskip
We develop a predictor-feedback control design for multi-input nonlinear systems with distinct input delays, of arbitrary length, in each individual input channel. Due to the fact that different input signals reach the plant at different time instants, the key design challenge, which we resolve, is the construction of the predictors of the plant's state over distinct prediction horizons such that the corresponding input delays are compensated. Global asymptotic stability of the closed-loop system is established by utilizing arguments based on Lyapunov functionals or estimates on solutions. We specialize our methodology to linear systems for which the predictor-feedback control laws are available explicitly and for which global exponential stability is achievable. A detailed example is provided dealing with the stabilization of the nonholonomic unicycle, subject to two different input delays affecting the speed and turning rate, for the illustration of our methodology.
\end{abstract}
\baselineskip=1.8\normalbaselineskip
\section{Introduction}
\paragraph{Background and Motivation}






Despite the recent outburst in the development of predictor-based control laws for nonlinear systems with input delays \cite{bek1}, \cite{bek2}, \cite{bek3}, \cite{bek4}, \cite{bek5}, \cite{bek6}, \cite{bek7}, \cite{delph1}, \cite{delph2}, \cite{delph3}, \cite{delph4}, \cite{delph5}, \cite{kar1}, \cite{kar2}, \cite{kar3}, \cite{kar4}, \cite{kar5}, \cite{karn1}, \cite{kar6}, \cite{krstic first}, \cite{krstic book}, \cite{krstic}, \cite{mazenc0}, \cite{mazenc1}, \cite{mazenc11}, \cite{mazenc2}, \cite{mazenc3}, the problem of the systematic predictor-feedback stabilization of multi-input nonlinear systems with, potentially different, in each individual input channel, long input delays, has remained, heretofore, untackled, although the problem was solved in the linear case in the early 1980s \cite{artstein} (see also \cite{manitius}).  In this article, we address the problem of stabilization of multi-input nonlinear systems with distinct input delays of arbitrary length and develop a nonlinear version of the prediction-based control laws developed in \cite{artstein} and recently in \cite{daisuke new}, \cite{daisuke} for the compensation of input delays in multi-input linear systems.

Besides the unavailability of a systematic predictor-feedback design methodology for multi-input nonlinear systems with long input delays, the real motivation for this article comes from applications. Such systems serve as models for the dynamics of traffic \cite{ge}, \cite{orosz2}, teleoperators \cite{spong} and robotic manipulators \cite{un}, \cite{fisher}, motors \cite{farshad}, \cite{gideon}, multi-agent systems \cite{multi agent 1}, \cite{multi agent 2}, \cite{malisoff2}, autonomous ground vehicles \cite{malisoff1}, unmanned aerial vehicles \cite{gru} and planar vertical take-off and landing aircrafts \cite{mazencvtol}, \cite{palomino}, and the human musculoskeletal system in applications such as neuromuscular electrical stimulation \cite{kar6}, \cite{lan}, \cite{dix1}, to name only a few. Motivated by the negative effects of input delays on the stability and performance of such control systems, in this article we present control designs that achieve delay compensation.

\paragraph{Contributions}

We introduce a predictor-feedback control design for the compensation of long input delays in multi-input nonlinear systems. Since each individual input channel might induce a different delay the predictors of the plant's state are constructed recursively starting from the predictor that corresponds to the smallest input delay all the way through to the predictor that corresponds to the largest input delay. Specifically, at each step, the predictor, over the prediction horizon that corresponds to the current's step input delay, is constructed by actually predicting, over the appropriate prediction window, the future values of the predictor constructed at the previous step.

We conduct the stability analysis of the closed-loop system, under the developed predictor-feedback control law, utilizing two different techniques--one based on the construction of a Lyapunov functional and one based on estimates on the solutions of the closed-loop system. In the former case, the construction of a Lyapunov functional is enabled by the introduction of novel backstepping transformations of the actuator states, which are based on an equivalent, PDE representation of the constructed predictor states. In the latter approach, we exploit the fact that each delay is compensated after a finite time.

We present a detailed example, including numerical simulations, dealing with the stabilization of a nonholonomic robot subject to different input delays, in order to highlight the intricacies of our design and analysis methodologies. We specialize our results to the case of linear systems for which the predictor-feedback control laws are obtained explicitly and for which global exponential stability is achievable.








\paragraph{Organization}
We start in Section \ref{sec2} with the introduction to the problem of predictor-feedback stabilization of multi-input nonlinear systems and develop the predictor-feedback control laws. In Section \ref{sec3} we prove global asymptotic stability of the closed-loop system under predictor-feedback by constructing a Lyapunov functional and in Section \ref{sec4} we prove global asymptotic stability using estimates on solutions. Section \ref{uni} is devoted to a detailed example of stabilization of the nonholonomic unicycle subject to input delays. We specialize our methodology to the case of linear systems in Section \ref{linear section}. For the example worked out in detail in Section \ref{uni} we present simulation results in Section \ref{sec sim}.

\paragraph*{Notation}We use the common definition of class $\mathcal{K}$, $\mathcal{K}_{\infty}$ and $\mathcal{KL}$ functions from \cite{khalil}. For an $n$-vector, the norm $|\cdot|$ denotes the usual Euclidean norm. For a function $u: [0,D]\times\mathbb{R}_+\to\mathbb{R}$ we denote by $\|u(t)\|_{\infty}$ its spatial supremum norm, i.e., $\|u(t)\|_{\infty}=\sup_{x\in[0,D]}|u(x,t)|$. For any $c>0$, we denote the spatially weighted supremum norm of $u$ by $\|u(t)\|_{c,\infty}=\sup_{x\in[0,D]}|e^{cx}u(x,t)|$. For a vector valued function $p: [0,D]\times \mathbb{R}_+\to\mathbb{R}^n$ we use a spatial supremum norm $\|p(t)\|_{\infty}=\sup_{x\in[0,D]}\sqrt{p_1(x,t)^2+\ldots p_n(x,t)^2}$. We denote by $C^j(A;E)$ the space of functions that take values in $E$ and have continuous derivatives of order $j$ on $A$. 

\section{Multi-Input Nonlinear Systems with Distinct Delays \\and Predictor-Feedback Control Design}
\label{sec2}
We consider the following system (see Fig. \ref{multifig})
\begin{figure}
\centering
\includegraphics[width=0.8\linewidth]{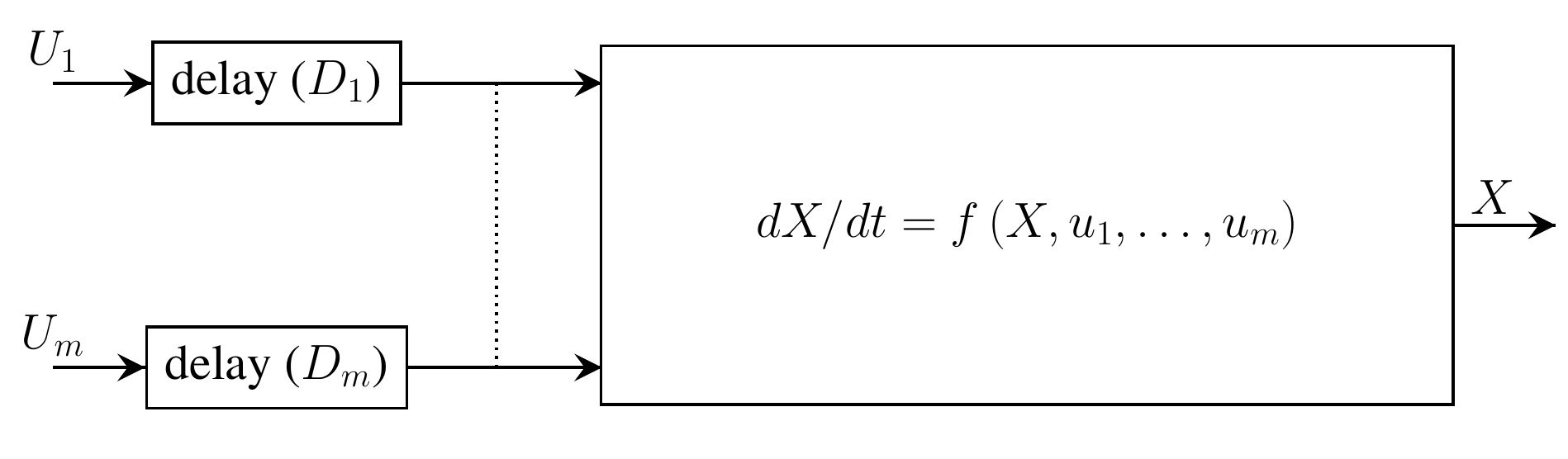}
\caption{Multi-input nonlinear system with distinct input delays.}
\label{multifig}
\end{figure}
\begin{eqnarray}
\dot{X}(t)=f\left(X(t),U_1\left(t-D_1\right),\ldots,U_m\left(t-D_m\right)\right),\label{plant}
\end{eqnarray}
where $X\in\mathbb{R}^n$ is state, $t\geq0$ is time, $U_1 ,\ldots,U_m\in\mathbb{R}$ are control inputs, $D_1,\ldots,D_m$ are (potentially distinct) input delays satisfying (without loss of generality) $0<D_1\leq \cdots\leq D_m$, and $f:\mathbb{R}^n\times\mathbb{R}^m\to\mathbb{R}^n$ is a locally Lipschitz vector field that satisfies $f(0,0,\ldots,0)=0$. The predictor feedback controllers are defined by
\begin{eqnarray}
U_i(t)&=&\kappa_i\left(P_i(t)\right)\label{con1},\quad i=1,\ldots,m,\label{contr1}
\end{eqnarray}
where $\kappa_i:\mathbb{R}^n\to\mathbb{R}$, $i=1,\ldots,m$, are continuously differentiable feedback laws with $\kappa_i(0)=0$, $i=1,\ldots,m$, and, $P_i$ are the $D_i$-time units ahead predictors of $X$, for all $i=1,\ldots,m$. Defining $D_{ji}=D_j-D_i$, for all $i\leq j\leq m$, the predictors are given by
\begin{eqnarray}
P_1(t)&=&X(t)+\int_{t-D_1}^tf\left(P_1(\theta),U_1(\theta),U_2(\theta-D_{21}),\ldots,U_m(\theta-D_{m1})\right)d\theta\label{P11}\\
P_2(t)&=&P_1(t)+\int_{t-D_{21}}^tf\left(P_2(\theta),\kappa_1\left(P_2(\theta)\right),U_2(\theta),U_3(\theta-D_{32}),\ldots,U_m(\theta-D_{m2})\right)d\theta\label{conn}\\
&\vdots&\nonumber\\
P_m(t)&=&P_{m-1}(t)+\int_{t-D_{mm-1}}^tf\left(P_m(\theta),\kappa_1\left(P_m(\theta)\right),\kappa_2\left(P_m(\theta)\right),\ldots,\kappa_{m-1}\left(P_m(\theta)\right),U_m(\theta)\right)d\theta\label{conn1},
\end{eqnarray}
with initial conditions for the integral equations (\ref{P11})--(\ref{conn1})
\begin{eqnarray}
P_1(\theta)&=&X(0)+\int_{-D_1}^{\theta}f\left(P_1(s),U_1(s),U_2(s-D_{21}),\ldots,U_m(s-D_{m1})\right)ds,\quad-D_1\leq\theta\leq0\label{P11in}\\
P_2(\theta)&=&P_1(0)+\int_{-D_{21}}^{\theta}f\left(P_2(s),\kappa_1\left(P_2(s)\right),U_2(s),U_3(s-D_{32}),\ldots,U_m(s-D_{m2})\right)ds,\nonumber\\
&& -D_{21}\leq\theta\leq0\label{connin}\\
&\vdots&\nonumber\\
P_m(\theta)&=&P_{m-1}(0)+\int_{-D_{mm-1}}^{\theta}f\left(P_m(s),\kappa_1\left(P_m(s)\right),\kappa_2\left(P_m(s)\right),\ldots,\kappa_{m-1}\left(P_m(s)\right),U_m(s)\right)ds,\nonumber\\
&&-D_{mm-1}\leq\theta\leq0\label{conn1in}.
\end{eqnarray}

We show that $P_i$, for all $i=1,\ldots,m$, are the $D_i$-time units ahead predictors of $X$ by induction. In order to better understand the general induction step we provide two initial steps. We show first that $P_1$ and $P_2$ are the $D_1$- and $D_2$-time units ahead predictors of $X$, respectively. 

\subsubsection*{Step 1}
We perform the change of variables $t=\theta+D_1$, for all $\theta\geq -D_1$, in (\ref{plant}) and define $P_1(\theta)=X(\theta+D_1)$, $\theta\geq -D_1$, to arrive at
\begin{eqnarray}
\frac{dP_1(\theta)}{d\theta}=f\left(P_1(\theta),U_1(\theta),U_2(\theta-D_{21}),\ldots,U_m(\theta-D_{m1})\right),\quad \mbox{for all $\theta\geq -D_1$}.\label{odep1}
\end{eqnarray}
Integrating (\ref{odep1}) from $\theta=t-D_1$ to $\theta=t$ and using definition $P_1(\theta)=X(\theta+D_1)$ we get (\ref{P11}). Integrating (\ref{odep1}) from $\theta=-D_1$ to any $\theta\leq0$ and using definition $P_1(-D_1)=X(0)$ we get (\ref{P11in}).

\subsubsection*{Step 2}
Performing the change of variables $\theta=s+D_{21}$, for all $s\geq -D_{21}$, in (\ref{odep1}) and defining $P_2(s)=P_1(s+D_{21})=X(s+D_2)$, for all $s\geq-D_{21}$, we get that
\begin{eqnarray}
\frac{dP_2(s)}{ds}=f\left(P_2(s),\kappa_1\left(P_2(s)\right),U_2(s),U_3(s-D_{32}),\ldots,U_m(s-D_{m2})\right),\quad \mbox{for all $s\geq -D_{21}$},\label{odep2}
\end{eqnarray}
where we also used the fact that $U_1(s+D_{21})=\kappa_1\left(P_1(s+D_{21})\right)$, for all $s+D_{21}\geq0$, and definition  $P_2(s)=P_1(s+D_{21})$. Integrating (\ref{odep2}) from $s=t-D_{21}$ to $s=t$ and using definition $P_2(s)=P_1(s+D_{21})$, for all $s\geq-D_{21}$, we arrive at (\ref{conn}).  Integrating (\ref{odep2}) from $\theta=-D_{21}$ to any $\theta\leq0$ and using definition $P_2(-D_{21})=P_1(0)$ we get (\ref{connin}).

\subsubsection*{Step j}
Assume now that the $D_j$-time units ahead predictor of $X$, namely $P_j$, satisfies the following ODE in $r$
\begin{eqnarray}
\frac{dP_j(r)}{dr}&=&f\left(P_j(s),\kappa_1\left(P_j(r)\right),\ldots,\kappa_{j-1}\left(P_j(r)\right),U_j(r),U_{j+1}(r-D_{j+1j}),\ldots,U_m(r-D_{mj})\right),\nonumber\\
&& \mbox{for all $r\geq -D_{jj-1}$}.\label{odepj}
\end{eqnarray}
Performing the change of variables $r=h+D_{j+1j}$, for all $h\geq -D_{j+1j}$, in (\ref{odepj}) and defining $P_{j+1}(h)=P_j(h+D_{j+1j})=X(h+D_{j+1})$, for all $h\geq-D_{j+1j}$, we get that
\begin{eqnarray}
\frac{dP_{j+1}(h)}{dh}&=&f\left(P_{j+1}(h),\kappa_1\left(P_{j+1}(h)\right),\kappa_j\left(P_{j+1}(h)\right),U_{j+1}(h),\ldots,U_m(h-D_{mj+1})\right),\nonumber\\
&& \mbox{for all $h\geq -D_{j+1j}$},\label{odeph}
\end{eqnarray}
where we also used the fact that $U_j(h+D_{j+1j})=\kappa_j\left(P_j(h+D_{j+1j})\right)$, for all $h+D_{j+1j}\geq0$, and definition $P_{j+1}(h)=P_j(h+D_{j+1j})$. Integrating (\ref{odeph}) from $h=t-D_{j+1j}$ to $h=t$ and from $h=-D_{j+1j}$ to any $h\leq0$, and using definition $P_{j+1}(h)=P_j(h+D_{j+1j})$, for all $h\geq-D_{j+1j}$ (which implies that  $P_{j+1}(-D_{j+1j})=P_j(0)$), we conclude that indeed the $D_i$-time units ahead predictors of $X$, for all $i=1,\ldots,m$, are given by (\ref{P11})--(\ref{conn1}) with initial conditions (\ref{P11in})--(\ref{conn1in}).

\section{Lyapunov-Based Stability Analysis Under Predictor Feedback}
\label{sec3}
\begin{assumption}
\label{ass1}
The system $\dot{X}=f\left(X,\omega_1,\ldots,\omega_m\right)$ is strongly forward complete with respect to $\omega=\left(\omega_1,\ldots,\omega_m\right)^T$.
\end{assumption}
\begin{assumption}
\label{ass3}
The system $\dot{X}=f\left(X,\omega_1+\kappa_1(X),\ldots,\omega_m+\kappa_m(X)\right)$ is input-to-state stable with respect to $\omega=\left(\omega_1,\ldots,\omega_m\right)^T$.
\end{assumption}
\begin{assumption}
\label{ass2}
The systems $\dot{X}=g_j\left(X,\omega_{j+1},\ldots,\omega_m\right)$, for all $j=1,\ldots,m-1$, with $g_j\left(X,\omega_{j+1},\ldots,\omega_m\right)=f\left(X,\kappa_1\left(X\right),\ldots,\kappa_{j}\left(X\right),\omega_{j+1},\ldots,\omega_m\right)$, are strongly forward complete with respect to $\omega=\left(\omega_{j+1},\ldots,\omega_m\right)^T$.
\end{assumption}

The definitions of strong forward completeness and input-to-state stability are those from \cite{krstic} (see also \cite{angeli} for the definition of standard forward completeness which differs from strong forward completeness in that $f\left(0,0,\ldots,0\right)=0$) and \cite{sontag}, respectively. 

Assumption \ref{ass1} guarantees that for every initial condition and every locally bounded input signal the corresponding solution is defined for all $t\geq0$. In particular, the plant does not exhibit finite escape before the first feedback control reaches it. This is a natural requirement for achieving global stabilization in the presence of arbitrary large delays affecting the inputs of a system. Assumption \ref{ass3} can be relaxed to only global asymptotic stability of system $\dot{X}=f\left(X,\kappa_1(X),\ldots,\kappa_m(X)\right)$. Yet, at the expense of not having a Lyapunov functional available. Assumption \ref{ass2} guarantees that after the $j$-th controller ``kicks in" and the $D_j$-th delay is compensated, and hence, the plant behaves according to $\dot{X}=f\left(X,\kappa_1(X),\ldots,\kappa_j\left(X\right),U_{j+1}(t-D_{j+1}),\ldots,U_m(t-D_m)\right)$, the solutions are also well-defined. In particular, the plant does not exhibit finite escape before the $j+1$-th feedback control reaches it and after the $j$-th feedback control has already reached the plant. Note that Assumption \ref{ass2} can be relaxed to strong forward completeness of systems $\dot{X}=g_j\left(X,\omega_{j+1},\ldots,\omega_m\right)$ with respect to $\omega=\left(\omega_{j+1},\ldots,\omega_m\right)^T$, for all $j\in\left\{r_1,r_1+r_2,\ldots,r_1+\ldots+r_{\nu}\right\}$, where $g_j\left(X,\omega_{j+1},\ldots,\omega_m\right)=f\left(X,\kappa_1\left(X\right),\ldots,\kappa_{j}\left(X\right),\omega_{j+1},\ldots,\omega_m\right)$, $r_1$ denotes the number of delays that are equal to $D_1$, $r_{\sigma}$, $\sigma=2,\ldots,\nu$, denotes the number of delays that are equal to $D_{r_1+\ldots+r_{\sigma-1}}$, and $\nu$ is the number of distinct delays. In particular, when all delays are identical, Assumption \ref{ass2} can be completely removed.


The stability proof is based on an equivalent representation of plant (\ref{plant}), using transport PDEs for the actuator states, and on an equivalent PDE representation of the predictor states (\ref{P11})--(\ref{conn1}). We present the alternative representations for the plant and the predictor states before stating and proving the main result of this section, since the reader might find the alternative formalisms helpful in better digesting the design and analysis ideas of our methodology.
\subsection{Equivalent Representation of the Plant Using Transport PDEs for the Actuator States}
System (\ref{plant}) can be written equivalently as 
\begin{eqnarray}
\dot{X}(t)&=&f\left(X(t),u_1(0,t),\ldots,u_m(0,t)\right)\label{plant1}\\
\partial_t u_i(x,t)&=&\partial_x u_i(x,t),\quad x\in(0,D_i),\quad i=1,\ldots,m\label{pde1}\\
u_i\left(D_i,t\right)&=&U_i(t),\quad i=1,\ldots,m\label{pde2}.
\end{eqnarray}
To see this note that the solutions to (\ref{pde1}), (\ref{pde2}) are given by 
\begin{eqnarray}
u_i(x,t)=U_i(t+x-D_i),\quad x\in[0,D_i],\quad i=1,\ldots,m.\label{solu}
\end{eqnarray}
\subsection{Transport PDE Representation of the Predictor States}
\label{sudpre}
The predictor states $P_1(\theta)$, for all $\theta\geq -D_1$, and $P_j(\theta)$, for all $\theta\geq -D_{jj-1}$ and $j=2,\ldots,m$, can be written equivalently as (see Fig. \ref{figdiag})
\begin{figure}
\centering
\includegraphics[width=\linewidth]{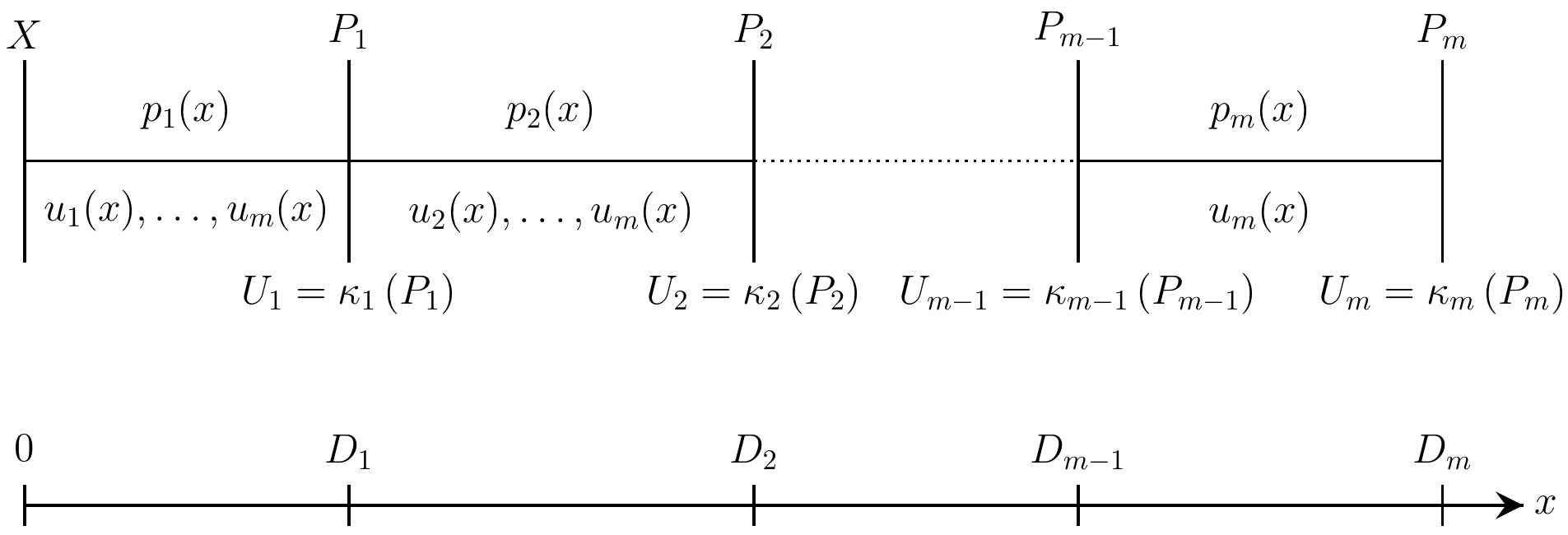}
\caption{The $D_j$-time units ahead predictors of the state $X$, namely $P_j$, given in (\ref{P11})--(\ref{conn1}), and their equivalent representation by the PDE states $p_j$, given in (\ref{p1})--(\ref{pm}), based on the transport-PDE equivalent of the actuator states defined in (\ref{pde1})--(\ref{pde2}). The control laws $U_j$ are defined in (\ref{contr1}) in terms of $P_j$ and can be written equivalently as in (\ref{connew1}) in terms of $p_j$.}
\label{figdiag}
\end{figure}
\begin{eqnarray}
p_1(x,t)&=&X(t)+\int_0^xf\left(p_1(y,t),u_1(y,t),\ldots,u_m(y,t)\right)dy,\quad x\in[0,D_1]\label{p1}\\
p_2(x,t)&=&p_1\left(D_1,t\right)+\int_{D_1}^xf\left(p_2(y,t),\kappa_1\left(p_2(y,t)\right),u_2(y,t),\ldots,u_m(y,t)\right)dy,\quad x\in[D_1,D_2]\label{p2}\\
&\vdots&\nonumber\\
p_m(x,t)&=&p_{m-1}\left(D_{m-1},t\right)+\int_{D_{m-1}}^xf\left(p_m(y,t),\kappa_1\left(p_m(y,t)\right),\ldots,\kappa_{m-1}\left(p_m(y,t)\right),u_m(y,t)\right)dy,\nonumber\\
&& x\in[D_{m-1},D_m].\label{pm}
\end{eqnarray}
We show this by induction. In order to make the presentation of the procedure clearer we present two steps before the general step. We first observe that (see Section \ref{sec2}) $P_1(\theta)=X(\theta+D_1)$, for all $\theta\geq -D_1$, and $P_j(\theta)=X(\theta+D_j)$, for all $\theta\geq-D_{jj-1}$ and $j=2,\ldots,m$. The functions $p_i$, $i=1,\ldots,m$, satisfy the following ODEs in $x$
\begin{eqnarray}
\partial_xp_1(x,t)&=&f\left(p_1(x,t),u_1(x,t),\ldots,u_m(x,t)\right),\quad x\in[0,D_1]\label{p1ode}\\
\partial_xp_2(x,t)&=&f\left(p_2(x,t),\kappa_1\left(p_2(x,t)\right),u_2(x,t),\ldots,u_m(x,t)\right),\quad x\in[D_1,D_2]\label{p2ode}\\
&\vdots&\nonumber\\
\partial_xp_m(x,t)&=&f\left(p_m(x,t),\kappa_1\left(p_m(x,t)\right),\ldots,\kappa_{m-1}\left(p_m(x,t)\right),u_m(x,t)\right),\quad x\in[D_{m-1},D_m],\label{pmode}
\end{eqnarray}
with initial conditions
\begin{eqnarray}
p_1(0,t)&=&X(t)\label{p1b}\\
p_2(D_1,t)&=&p_1(D_1,t)\label{p2b}\\
&\vdots&\nonumber\\
p_m(D_{m-1},t)&=&p_{m-1}(D_{m-1},t).\label{pmb}
\end{eqnarray}

\subsubsection*{Step 1}
The solution to (\ref{p1ode}), (\ref{p1b}) is 
\begin{eqnarray}
p_1(x,t)=X(t+x),\quad x\in[0,D_1].\label{sol1}
\end{eqnarray}
In order to show this, first note that (\ref{sol1}) satisfies the boundary condition (\ref{p1b}). The function $X(t+x)$ also satisfies the ODE in $x$ (\ref{p1ode}) which follows from the fact that by (\ref{plant}) one can conclude that
\begin{eqnarray}
X'(t+x)=f\left(X(t+x),U_1(t+x-D_1),\ldots,U_m(t+x-D_m)\right),\quad\mbox{for all $t\geq0$ and $0\leq x\leq D_1$}.
\end{eqnarray}
The result follows from the uniqueness of solutions to the ODE (\ref{plant}). Therefore, by defining 
\begin{eqnarray}
p_1(D_1,t)&=&P_1(t),\label{25}
\end{eqnarray}
and using the fact that $p_1$ is a function of one variable, namely $x+t$ (which follows from (\ref{sol1})), one can conclude that
\begin{eqnarray}
P_1(t+x-D_1)&=&p_1(x,t),\quad x\in[0,D_1]\label{defP}.
\end{eqnarray}
Performing the change of variables $x=\theta+D_1-t$, for all $t-D_1\leq\theta\leq t$, in (\ref{p1}) and using (\ref{solu}), (\ref{p1b}), and (\ref{defP}) we arrive at
\begin{eqnarray}
P_1(\theta)&=&X(t)+\int_{t-D_1}^{\theta}f\left(P_1(s),U_1(s),\ldots,U_m(s-D_{m1})\right)ds,\quad\mbox{for all $t-D_1\leq\theta\leq t$}.
\end{eqnarray}

\subsubsection*{Step 2}
Similarly, it can be shown that  
\begin{eqnarray}
p_2(x,t)=X(t+x),\quad x\in[D_1,D_2],\label{sol2}
\end{eqnarray}
is the solution to (\ref{p2ode}), (\ref{p2b}) since it satisfies (\ref{p2b}) and since it satisfies the ODE in $x$ (\ref{p2ode}) which follows from the fact that the function $X(t+x)$ satisfies 
\begin{eqnarray}
X'(t+x)&=&f\left(X(t+x),\kappa_1\left(X(t+x)\right),U_2(t+x-D_2),\ldots,U_m(t+x-D_m)\right),\nonumber\\
&&\mbox{for all $t\geq0$ and $D_1\leq x\leq D_2$},
\end{eqnarray}
where we also used the fact that $U_1(t+x-D_1)=\kappa_1\left(P_1(t+x-D_1)\right)=\kappa_1\left(X(t+x)\right)$, for all $x\geq D_1$. Defining 
\begin{eqnarray}
p_2(D_2,t)&=&P_2(t),
\end{eqnarray}
and using the fact that $p_2$ is a function of one variable, namely $x+t$ (which follows from (\ref{sol2})), one can conclude that
\begin{eqnarray}
P_2(t+x-D_2)&=&p_2(x,t),\quad x\in[D_1,D_2]\label{defP2}.
\end{eqnarray}
Performing the change of variables $x=\theta+D_2-t$, for all $t-D_{21}\leq\theta\leq t$, in (\ref{p2}) and using (\ref{solu}), (\ref{25}), and (\ref{defP2}) we arrive at
\begin{eqnarray}
P_2(\theta)&=&P_1(t)+\int_{t-D_{21}}^{\theta}f\left(P_2(s),\kappa_1\left(P_2(s)\right),U_2(s),U_3(s-D_{32}),\ldots,U_m(s-D_{m2})\right)ds,\nonumber\\
&&\mbox{for all $t-D_{21}\leq\theta\leq t$}.
\end{eqnarray}

\subsubsection*{Step j}
In general, assume that for some $j$
\begin{eqnarray}
p_j(x,t)&=&P_j(t+x-D_j)\nonumber\\
&=&X(t+x),\quad x\in[D_{j-1},D_j]\label{solj}.
\end{eqnarray}
We show next that $p_{j+1}(x,t)=X(t+x)$, for all $x\in[D_j,D_{j+1}]$. We first observe that $p_{j+1}(D_j,t)=X(t+D_j)=p_j(D_{j},t)$. Moreover, the function $p_{j+1}(x,t)=X(t+x)$, for all $x\in[D_j,D_{j+1}]$, satisfies the following ODE in $x$
\begin{eqnarray}
\partial_xp_{j+1}(x,t)&=&f\left(p_{j+1}(x,t),\kappa_1\left(p_{j+1}(x,t)\right),\ldots,\kappa_{j}\left(p_{j+1}(x,t)\right),u_{j+1}(x,t),\ldots,u_m(x,t)\right),\nonumber\\
&&\quad x\in[D_{j},D_{j+1}],\label{pm1ode}
\end{eqnarray}
since the following holds
\begin{eqnarray}
X'(t+x)&=&f\left(X(t+x),\kappa_1\left(X(t+x)\right),\ldots,\kappa_j\left(X(t+x)\right),\right.\nonumber\\
&&\left.U_{j+1}(t+x-D_{j+1}),\ldots,U_m(t+x-D_m)\right),\quad\mbox{for all $t\geq0$ and $D_j\leq x\leq D_{j+1}$},
\end{eqnarray}
where we also used the fact that $U_i(t+x-D_i)=P_i(t+x-D_i)=X(t+x)$, for all $x\geq D_j$ and $i\leq j$. By defining 
\begin{eqnarray}
p_{j+1}(D_{j+1},t)&=&P_{j+1}(t),
\end{eqnarray}
once can conclude that 
\begin{eqnarray}
P_{j+1}(t+x-D_{j+1})&=&p_{j+1}(x,t),\quad x\in[D_{j},D_{j+1}].
\end{eqnarray}
Performing the change of variables $x=\theta+D_{j+1}-t$, for all $t-D_{j+1}\leq\theta\leq t$, we arrive at
\begin{eqnarray}
P_{j+1}(\theta)&=&P_{j}(t)+\int_{t-D_{j+1j}}^{\theta}f\left(P_{j+1}(s),\kappa_1\left(P_{j+1}(s)\right),\ldots,\kappa_{j}\left(P_j(s)\right),U_{j+1}(s),\right.\nonumber\\
&&\left.U_{j+2}(s-D_{j+2j+1}),\ldots,U_m(s-D_{mj+1})\right)ds,\quad\mbox{for all $t-D_{jj+1}\leq\theta\leq t$},
\end{eqnarray}
which completes the proof.

Note that with this representation we have that
\begin{eqnarray}
U_i(t)&=&\kappa_i\left(p_i(D_i,t)\right),\quad i=1,\ldots,m\label{connew1}.
\end{eqnarray}

\subsection{Main Result and its Proof}
\begin{theorem}
\label{theorem1}
Consider the closed-loop system consisting of the plant (\ref{plant1})--(\ref{pde2}) and the control laws (\ref{connew1}), (\ref{p1})--(\ref{pm}). Under Assumptions \ref{ass1}, \ref{ass3}, and \ref{ass2}, there exists a class $\mathcal{KL}$ function $\beta$ such that for all initial conditions $X_0\in\mathbb{R}^n$ and $u_{i_0}\in C[0,D_i]$, $i=1,\ldots,m$, which are compatible with the feedback laws, the closed-loop system has a unique solution $X(t)\in C^1[0,\infty)$ and $u_i(x,t)\in C\left([0,D_i]\times [0,\infty)\right)$, $i=1,\ldots,m$, and the following holds
\begin{eqnarray}
\Xi(t)\leq{\beta}\left(\Xi(0),t\right), \quad\mbox{for all $t\geq0$},\label{estimate theorem 1}
\end{eqnarray}
where
\begin{eqnarray}
\Xi(t)=\left|X(t)\right|+\sum_{i=1}^{i=m}\|u_i(t)\|_{\infty}.\label{norm}
\end{eqnarray}
\end{theorem}

\begin{corollary}[The version of Theorem \ref{theorem1} in standard delay notation]
\label{cor1}
Consider the closed-loop system consisting of the plant (\ref{plant}) and the control laws (\ref{contr1})--(\ref{conn1in}). Under Assumptions \ref{ass1}, \ref{ass3}, and \ref{ass2}, the following holds
\begin{eqnarray}
\left|X(t)\right|+\sum_{i=1}^{i=m}\sup_{t-D_i\leq\theta\leq t}|U_i(\theta)|\leq{\beta}\left(\left|X(0)\right|+\sum_{i=1}^{i=m}\sup_{-D_i\leq\theta\leq 0}|U_i(\theta)|,t\right), \quad\mbox{for all $t\geq0$}.\label{estimate cor 1}
\end{eqnarray}
\end{corollary}

The proof of Theorem \ref{theorem1} is based on a series of technical lemmas which are presented next and whose proofs are provided in Appendix A. Corollary \ref{cor1} follows immediately from Theorem \ref{theorem1} by using (\ref{solu}).

\begin{lemma}
\label{leminvz}
The backstepping transformations of $u_i$, $i=1,\ldots,m$, defined by
 \begin{eqnarray}
w_1(x,t)&=&u_1(x,t)-\kappa_1\left(p_1(x,t)\right),\quad x\in[0,D_1]\label{eq2}\\
w_2(x,t)&=&u_2(x,t)-\left\{\begin{array}{ll}\kappa_2\left(p_1(x,t)\right),&x\in[0,D_1]\\\kappa_2\left(p_2(x,t)\right),&x\in[D_1,D_2]\end{array}\right.\label{bac2new}\\
&\vdots&\nonumber\\
w_m(x,t)&=&u_m(x,t)-\left\{\begin{array}{ll}\kappa_m\left(p_1(x,t)\right),&x\in[0,D_1]\\\kappa_m\left(p_2(x,t)\right),&x\in[D_1,D_2]\\\vdots\\\kappa_m\left(p_m(x,t)\right),&x\in[D_{m-1},D_m]\end{array}\right.,\label{bacmnew}
\end{eqnarray}
where $p_i$, $i=1,\ldots,m$, are defined in (\ref{p1})--(\ref{pm}), together with the control laws (\ref{connew1}), (\ref{p1})--(\ref{pm}), transform system (\ref{plant1})--(\ref{pde2}) to the following ``target system"
\begin{eqnarray}
\dot{X}(t)&=&f\left(X(t),w_1(0,t)+\kappa_1\left(X(t)\right),\ldots,w_m(0,t)+\kappa_m\left(X(t)\right)\right)\label{plant1tr}\\
\partial_t w_i(x,t)&=&\partial_x w_i(x,t),\quad x\in(0,D_i),\quad i=1,\ldots,m\label{pde1tr}\\
w_i\left(D_i,t\right)&=&0,\quad i=1,\ldots,m\label{pde2tr}.
\end{eqnarray}
\end{lemma}

\begin{lemma}
\label{inversezeta}
The inverse backstepping transformations of (\ref{eq2})--(\ref{bacmnew}) are defined by
 \begin{eqnarray}
u_1(x,t)&=&w_1(x,t)+\kappa_1\left(\pi_1(x,t)\right),\quad x\in[0,D_1]\label{eq2in}\\
u_2(x,t)&=&w_2(x,t)+\left\{\begin{array}{ll}\kappa_2\left(\pi_1(x,t)\right),&x\in[0,D_1]\\\kappa_2\left(\pi_2(x,t)\right),&x\in[D_1,D_2]\end{array}\right.\label{bac2newin}\\
&\vdots&\nonumber\\
u_m(x,t)&=&w_m(x,t)+\left\{\begin{array}{ll}\kappa_m\left(\pi_1(x,t)\right),&x\in[0,D_1]\\\kappa_m\left(\pi_2(x,t)\right),&x\in[D_1,D_2]\\\vdots\\\kappa_m\left(\pi_m(x,t)\right),&x\in[D_{m-1},D_m]\end{array}\right.,\label{bacmnewinv}
\end{eqnarray}
where
\begin{eqnarray}
\pi_1(x,t)&=&X(t)+\int_0^xf\left(\pi_1(y,t),w_1(y,t)+\kappa_1\left(\pi_1(y,t)\right),\ldots,w_m(y,t)+\kappa_m\left(\pi_1(y,t)\right)\right)dy,\nonumber\\
&&\quad x\in[0,D_1]\\
\pi_2(x,t)&=&\pi_1(D_1,t)+\int_{D_1}^xf\left(\pi_2(y,t),\kappa_1\left(\pi_2(y,t)\right),w_2(y,t)+\kappa_2\left(\pi_2(y,t)\right),\right.\nonumber\\
&&\left.\ldots,w_m(y,t)+\kappa_m\left(\pi_1(y,t)\right)\right)dy,\quad x\in [D_1,D_2]\label{eqn}\\
&\vdots&\nonumber\\
\pi_m(x,t)&=&\pi_{m-1}(D_{m-1},t)+\int_{D_{m-1}}^xf\left(\pi_m(y,t),\kappa_1\left(\pi_m(y,t)\right),\ldots,\kappa_{m-1}\left(\pi_m(y,t)\right),\right.\nonumber\\
&&\left.w_m(y,t)+\kappa_m\left(\pi_m(y,t)\right)\right)dy,\quad x\in [D_{m-1},D_m]\label{eqnm}.
\end{eqnarray}
\end{lemma}

\begin{lemma}
\label{lemcl}
There exists a class $\mathcal{KL}$ function $\beta_1$ such that the following holds 
\begin{eqnarray}
\bar{\Xi}(t)&\leq& \beta_1\left(\bar{\Xi}(0),t\right),\quad\mbox{for all $t\geq0$}\label{blem},
\end{eqnarray}
where
\begin{eqnarray}
\bar{\Xi}(t)&=&\left|X(t)\right|+\sum_{i=1}^{i=m}\|w_i(t)\|_{\infty}.\label{norm1}
\end{eqnarray}
\end{lemma}

\begin{lemma}
\label{lemmabdirect}
There exist class $\mathcal{K}_{\infty}$ functions $\rho_1,\ldots,\rho_m$ such that
\begin{eqnarray}
\|p_1(t)\|_{\infty}&\leq&\rho_1\left({\Xi}(t)\right)\label{pboundd}\\
\|p_2(t)\|_{\infty}&\leq&\rho_2\left({\Xi}(t)\right)\label{pboundd2}\\
&\vdots&\nonumber\\
\|p_m(t)\|_{\infty}&\leq&\rho_m\left({\Xi}(t)\right)\label{pbounddm},
\end{eqnarray}
where ${\Xi}$ is defined in (\ref{norm}).
\end{lemma}

\begin{lemma}
\label{lemmabdirectinv}
There exist class $\mathcal{K}_{\infty}$ functions $\bar{\rho}_1,\ldots,\bar{\rho}_m$ such that
\begin{eqnarray}
\|\pi_1(t)\|_{\infty}&\leq&\bar{\rho}_1\left(\bar{\Xi}(t)\right)\label{invbpi}\\
\|\pi_2(t)\|_{\infty}&\leq&\bar{\rho}_2\left(\bar{\Xi}(t)\right)\label{invbpi2}\\
&\vdots&\nonumber\\
\|\pi_m(t)\|_{\infty}&\leq&\bar{\rho}_m\left(\bar{\Xi}(t)\right)\label{invbpim},
\end{eqnarray}
where $\bar{\Xi}$ is defined in (\ref{norm1}).
\end{lemma}

\begin{lemma}
\label{lemmaeq}
There exist class $\mathcal{K}_{\infty}$ functions $\rho$, $\bar{\rho}$ such that
\begin{eqnarray}
\bar{\Xi}(t)&\leq&\rho\left(\Xi(t)\right)\label{b1eq}\\
\Xi(t)&\leq&\bar{\rho}\left(\bar{\Xi}(t)\right).\label{b2eq}
\end{eqnarray}
\end{lemma}

\begin*{{\it Proof of Theorem \ref{theorem1}}:}
Combining (\ref{b2eq}) with (\ref{blem}) we get that $\Xi(t)\leq\bar{\rho}\left( \beta_1\left(\bar{\Xi}(0),t\right)\right)$, for all $t\geq0$, and hence, with (\ref{b1eq}) we arrive at (\ref{estimate theorem 1}) with $\beta(s,t)=\bar{\rho}\left( \beta_1\left(\rho\left(s\right),t\right)\right)$. The proof of existence and uniqueness of a solution $X(t)\in C^1[0,\infty)$ and $u_i(x,t)\in C\left([0,D_i]\times [0,\infty)\right)$, $i=1,\ldots,m$, is shown as follows. Using relations (\ref{p1ode})--(\ref{pmode}) for $t=0$, the compatibility of the initial conditions $u_{i_0}$, $i=1,\ldots,m$, with the feedback laws (\ref{connew1}) guarantee that $p_i(x,0)\in C^1[D_{i-1},D_i]$, where $D_0=0$. Hence, using relations (\ref{p1})--(\ref{pm}), (\ref{p1b})--(\ref{pmb}), and the fact that $u_{i_0}\in C[0,D_i]$, $i=1,\ldots,m$, it also follows from (\ref{eq2})--(\ref{bacmnew}) that $w_{i_0}\in C[0,D_i]$, $i=1,\ldots,m$. The solution to (\ref{pde1tr}), (\ref{pde2tr}) is given by
\begin{eqnarray}
w_i(x,t)=\left\{\begin{array}{ll}w_{i_0}(t+x),&0\leq x+t<D_i\\0,&x+t\geq D_i\end{array}\right.,\quad i=1,\ldots,m\label{solution w12}.
\end{eqnarray}
The uniqueness of this solution follows from the uniqueness of the solution to (\ref{pde1tr}), (\ref{pde2tr}) (see Sections 2.1 and 2.3 in \cite{coron}). Hence, the compatibility of the initial conditions $u_{i_0}$, $i=1,\ldots,m$, with the feedback laws (\ref{connew1}) guarantee that there exist a unique solution $w_i(x,t)\in C\left([0,D_i]\times [0,\infty)\right)$, $i=1,\ldots,m$. From the target system (\ref{plant1tr}) it follows that $X(t)\in C^1[0,\infty)$. The fact that $\pi_i(x,t)=X(t+x)$, for all $t\geq0$ and $x\in [D_{i-1},D_i]$, and the inverse backstepping transformations (\ref{eq2in})--(\ref{bacmnewinv}) guarantee that $u_i(x,t)\in C\left([0,D_i]\times [0,\infty)\right)$, $i=1,\ldots,m$. The proof is completed. 
\end*

\section{Stability Analysis Under Predictor Feedback Using Estimates on Solutions}
\label{sec4}
\begin{theorem}
\label{theorem2}
Consider the closed-loop system consisting of the plant (\ref{plant}) and the control laws (\ref{con1})--(\ref{conn}). Let Assumptions \ref{ass1} and \ref{ass2} hold and assume that the system $\dot{X}=f\left(X,\kappa_1(X),\ldots,\kappa_m(X)\right)$ is globally asymptotically stable. There exists a class $\mathcal{KL}$ function $\hat{\beta}$ such that for all initial conditions $X_0\in\mathbb{R}^n$ and $U_{i_0}\in C[-D_i,0]$, $i=1,\ldots,m$, which are compatible with the feedback laws, there exists a unique solution to the closed-loop system with $X(t)\in C^1[0,\infty)$ and $U_i(t)$, $i=1,\ldots,m$, locally Lipschitz on $(0,\infty)$, and the following holds
\begin{eqnarray}
\Omega(t)\leq\hat{\beta}\left(\Omega(0),t\right),\quad \mbox{for all $t\geq0$},\label{estimate theorem 2}
\end{eqnarray}
where
\begin{eqnarray}
\Omega(t)=\left|X(t)\right|+\sum_{i=1}^{i=m}\sup_{t-D_i\leq\theta\leq t}\left|U_i(\theta)\right|.
\end{eqnarray}
\end{theorem}

\begin{proof}
We estimate first $|X(t)|$, for all $t\geq0$. Since the system $\dot{X}=f\left(X,\omega_1,\ldots,\omega_m\right)$ is forward complete using Lemma 3.5 from \cite{kaatrfyllis11} and the fact that $f(0,0,\ldots,0)=0$ (which allows us to set $R=0$), we get that 
\begin{eqnarray}
|X(t)|\leq \psi_1\left(|X(0)| +\sum_{i=1}^{i=m}\sup_{-D_i\leq\theta\leq 0}\left|U_i(\theta)\right|\right),\quad \mbox{for all $0\leq t\leq D_1$},\label{es1li}
\end{eqnarray}
for some class $\mathcal{K}_{\infty}$ function $\psi_1$. Using the fact that $U_1(t)=\kappa_1\left(P_1(t)\right)$, for all $t\geq0$, with $P_1(t)=X\left(t+D_1\right)$ and the fact that system $\dot{X}=f\left(X,\kappa_1(X),\omega_2,\ldots,\omega_m\right)$ is forward complete with respect to $\omega_2$ we get by applying Lemma 3.5 from \cite{kaatrfyllis11} that $|X(t)|\leq \bar{\psi}_2\left(|X(D_1)| +\sum_{i=2}^{i=m}\sup_{D_1-D_i\leq\theta\leq D_2-D_i}\left|U_i(\theta)\right|\right)$, for all $ D_1\leq t\leq D_2$, for some class $\mathcal{K}_{\infty}$ function $\bar{\psi}_2$. Hence, since $D_{j}\leq D_i$, $\forall j\leq i$, with (\ref{es1li}) we get
\begin{eqnarray}
|X(t)|\leq {\psi}_2\left(|X(0)| +\sum_{i=1}^{i=m}\sup_{-D_i\leq\theta\leq 0}\left|U_i(\theta)\right|\right),\quad \mbox{for all $ D_1\leq t\leq D_2$},\label{es2li} 
\end{eqnarray}
where the class $\mathcal{K}_{\infty}$ function ${\psi}_2(s)$ is defined as $\psi_2(s)=\bar{\psi}_2\left(\psi_1(s)+s\right)$. By repeatedly applying Lemma 3.5 from \cite{kaatrfyllis11} we get under Assumption \ref{ass2} that there exist class $\mathcal{K}_{\infty}$ functions $\bar{\psi}_j$, $j=3,\ldots,m$ such that $|X(t)|\leq \bar{\psi}_j\left(|X(D_{j-1})| +\sum_{i=j}^{i=m}\sup_{D_{j-1}-D_i\leq\theta\leq D_j-D_i}\left|U_i(\theta)\right|\right)$, for all $ D_{j-1}\leq t\leq D_j$, and hence,
\begin{eqnarray}
|X(t)|\leq \psi_j\left(|X(0)| +\sum_{i=1}^{i=m}\sup_{-D_i\leq\theta\leq 0}\left|U_i(\theta)\right|\right),\quad \mbox{for all $ D_{j-1}\leq t\leq D_j$},\label{esjli} 
\end{eqnarray}
where the class $\mathcal{K}_{\infty}$ functions ${\psi}_j$, $j=3,\ldots,m$ are defined as $\psi_j(s)=\bar{\psi}_{j}\left({\psi}_{j-1}(s)+s\right)$, $j=3,\ldots,m$, and where we also used the fact that $D_{j}\leq D_i$, $\forall j\leq i$. Combining (\ref{es1li})--(\ref{esjli}) we arrive at
\begin{eqnarray}
|X(t)|\leq \psi\left(|X(0)| +\sum_{i=1}^{i=m}\sup_{-D_i\leq\theta\leq 0}\left|U_i(\theta)\right|\right),\quad \mbox{for all $ 0\leq t\leq D_m$},\label{es3li} 
\end{eqnarray}
where $\psi(s)=\sum_{i=1}^{i=m}\psi_i(s)$. Using the fact that $U_m(t)=\kappa_m\left(P_m(t)\right)$, for all $t\geq0$, with $P_m(t)=X\left(t+D_m\right)$ and the fact that system $\dot{X}=f\left(X,\kappa_1(X),\ldots,\kappa_m(X)\right)$ is globally asymptotically stable we get that $|X(t)|\leq\hat{\beta}_1\left(|X(D_m)|,t-D_m\right)$, for all $t\geq D_m$, for some class $\mathcal{KL}$ function $\hat{\beta}_1$. Hence, using (\ref{es3li}) we get that
\begin{eqnarray}
|X(t)|\leq \hat{\beta}_2\left(\Omega(0),t\right),\quad \mbox{for all $ t\geq 0$},\label{es5li}
\end{eqnarray}
where the class $\mathcal{KL}$ function $\hat{\beta}_2$ is given by $\hat{\beta}_2(s,t)=\hat{\beta}_1\left(\psi(s),\max\left\{0,t-D_m\right\}\right)+\psi(s)e^{-\max\left\{0,t-D_m\right\}}$.

We estimate next $\sup_{t-D_1\leq\theta\leq t}\left|U_1(\theta)\right|$. Since $\kappa_1$ is locally Lipschitz and $\kappa_1(0)=0$, there exists a class $\mathcal{K}_\infty$ function $\hat{\alpha}_1$ such that $\left|\kappa_1(X)\right|\leq \hat{\alpha}_1\left(|X|\right)$, for all $X\in\mathbb{R}^n$. Since for all $t\geq0$ it holds that $U_1(t)=\kappa_1\left(X(t+D_1)\right)$ using (\ref{es5li}) one can conclude that
\begin{eqnarray}
\sup_{t-D_1\leq\theta\leq t}\left|U_1(\theta)\right|\leq\sup_{-D_1\leq\theta\leq 0}\left|U_1(\theta)\right|+\hat{\alpha}_1\left(\hat{\beta}_2\left(\Omega(0),0\right)\right),\quad0\leq t\leq D_1,
\end{eqnarray}
and hence,
\begin{eqnarray}
\sup_{t-D_1\leq\theta\leq t}\left|U_1(\theta)\right|\leq\hat{\alpha}_2\left(\Omega(0)\right),\quad 0\leq t\leq D_1,\label{estimate1}
\end{eqnarray}
where the function $\hat{\alpha}_2(s)=\hat{\alpha}_1\left(\hat{\beta}_2\left(s,0\right)\right)+s$ belongs to class $\mathcal{K}_{\infty}$. Using identical arguments we also get that
\begin{eqnarray}
\sup_{t-D_1\leq\theta\leq t}\left|U_1(\theta)\right|\leq\hat{\alpha}_1\left(\hat{\beta}_2\left(\Omega(0),t\right)\right),\quad t\geq D_1.\label{estimate2}
\end{eqnarray}
Combining (\ref{estimate1}) with (\ref{estimate2}) we arrive at
\begin{eqnarray}
\sup_{t-D_1\leq\theta\leq t}\left|U_1(\theta)\right|\leq\hat{\beta}_3\left(\Omega(0),t\right),\quad\mbox{for all $t\geq 0$},\label{estimate3}
\end{eqnarray}
where the class $\mathcal{KL}$ function $\hat{\beta}_3$ is defined by $\hat{\beta}_3(s,t)=\hat{\alpha}_1\left(\hat{\beta}_2\left(\Omega(0),\max\left\{0,t-D_1\right\}\right)\right)+\hat{\alpha}_2(s)e^{-\max\left\{0,t-D_1\right\}}$. Using the facts that $\kappa_2$ is locally Lipschitz and that $\kappa_2(0)=0$, which allows one to conclude that there exists a class $\mathcal{K}_\infty$ function $\hat{\alpha}_3$ such that $\left|\kappa_2(X)\right|\leq \hat{\alpha}_3\left(|X|\right)$, for all $X\in\mathbb{R}^n$, with similar arguments we get that
\begin{eqnarray}
\sup_{t-D_2\leq\theta\leq t}\left|U_2(\theta)\right|\leq\hat{\beta}_4\left(\Omega(0),t\right),\quad\mbox{for all $t\geq 0$},\label{estimate4}
\end{eqnarray}
where the class $\mathcal{KL}$ function $\hat{\beta}_4$ is defined by $\hat{\beta}_4(s,t)=\hat{\alpha}_3\left(\hat{\beta}_2\left(\Omega(0),\max\left\{0,t-D_2\right\}\right)\right)+\hat{\alpha}_4(s)e^{-\max\left\{0,t-D_2\right\}}$, where the function $\hat{\alpha}_4(s)=\hat{\alpha}_3\left(\hat{\beta}_2\left(s,0\right)\right)+s$ belongs to class $\mathcal{K}_{\infty}$. With the same arguments one can conclude that there exist class $\mathcal{KL}$ functions $\hat{\beta}_{j+2}$, $j=3,\ldots,m$, such that
\begin{eqnarray}
\sup_{t-D_j\leq\theta\leq t}\left|U_j(\theta)\right|\leq\hat{\beta}_{j+2}\left(\Omega(0),t\right),\quad j=3,\ldots,m, \quad\mbox{for all $t\geq 0$}.\label{estimate4j}
\end{eqnarray}
Combining estimates (\ref{es5li}), (\ref{estimate3}), (\ref{estimate4}), and (\ref{estimate4j}) we get (\ref{estimate theorem 2}) with $\hat{\beta}(s,t)=\hat{\beta}_2(s,t)+\sum_{i=1}^{i=m}\hat{\beta}_{i+2}(s,t)$. From (\ref{plant}) using the fact that $U_{i_0}\in C[-D_i,0]$, for all $i=1,\ldots,m$, the Lipschitzness of the vector field $f$ guarantees that $X(t)\in C^1[0,D_1)$. The fact that $U_1(t-D_1)=\kappa_1\left(X(t)\right)$, for all $t\geq D_1$, and the Lipschitzness of $\kappa_1$ guarantee that $X(t)\in C^1(D_1,D_2)$. Since $U_{1_0}$ is compatible with the first feedback law one can conclude that $X(t)\in C^1[0,D_2)$. Analogously, since for $t\geq D_1$ the state $X$ evolves according to $\dot{X}(t)=f\left(X(t),\kappa_1\left(X(t)\right),U_2(t-D_2),\ldots,U_m(t-D_m)\right)$, the fact that $U_{2_0}$ is compatible with the second feedback law and the fact that $U_2(t-D_2)=\kappa_2\left(X(t)\right)$, for all $t\geq D_2$, where $\kappa_2$ is locally Lipschitz, guarantee that $X(t)\in C^1[0,D_3)$. Continuing this procedure it is shown that $X(t)\in C^1[0,\infty)$. Since $U_i(t)=\kappa_i\left(X(t+D_i)\right)$, $i=1,\ldots,m$, the Lipshitzness of $\kappa_i$, $i=1,\ldots,m$, guarantees that $U_i(t)$, $i=1,\ldots,m$, are Lipschitz on $(0,\infty)$.
\end{proof}


\section{Stabilization of the Nonholonomic Unicycle Subject to Distinct Input Delays}
\label{uni}
We consider the following system
\begin{eqnarray}
\dot{X}_1(t)&=&U_2(t-D_2)\cos\left(X_3(t)\right)\label{sysex1}\\
\dot{X}_2(t)&=&U_2(t-D_2)\sin\left(X_3(t)\right)\\
\dot{X}_3(t)&=&U_1(t-D_1),\label{sysexn}
\end{eqnarray} 
which describes the dynamics of the unicycle, where $\left(X_1,X_2\right)$ is the position of the robot, $X_3$ is heading, $U_2$
is speed, and $U_1$ is the turning rate. We consider the following time-varying nominal control law designed in \cite{pomet}
\begin{eqnarray}
U_1(t)&=&-M(t)^2\cos\left(t\right)-M(t)Q(t)\left(1+\cos^2\left(t\right)\right)-X_3(t)\label{u1ex}\\
U_2(t)&=&-M(t)+Q(t)\left(\sin(t)-\cos(t)\right)+Q(t)U_1(t),\label{shit}
\end{eqnarray}
where
\begin{eqnarray}
M(t)&=&X_1(t)\cos\left(X_3(t)\right)+X_2(t)\sin\left(X_3(t)\right)\label{defr}\\
Q(t)&=&X_1(t)\sin\left(X_3(t)\right)-X_2(t)\cos\left(X_3(t)\right)\label{defq},
\end{eqnarray}
which achieves global asymptotic stabilization when $D_1=D_2=0$. We employ next our predictor-feedback design when $0<D_1<D_2$.  We first verify that Assumptions \ref{ass1} and \ref{ass2} are verified for system (\ref{sysex1})--(\ref{sysexn}) under the control laws (\ref{u1ex})--(\ref{defq}). We first note that system $\dot{X}=f\left(X,\omega_1,\omega_2\right)$, where
\begin{eqnarray}
f\left(X,\omega_1,\omega_2\right)=\left[\begin{array}{c}\omega_2\cos\left(X_3\right)\\\omega_2\sin\left(X_3\right)\\\omega_1\end{array}\right],
\end{eqnarray}
is forward complete with respect to $(\omega_1,\omega_2)$, and hence, Assumption \ref{ass1} is satisfied. One can see this by defining 
\begin{eqnarray}
R(X)=\frac{1}{2}\sum_{i=1}^{i=3}X_i^2,
\end{eqnarray}
and readily verifying by employing Young's inequality that 
\begin{eqnarray}
\frac{\partial R(X)}{\partial X}f\left(X,\omega_1,\omega_2\right)&=&X_1\omega_2\cos\left(X_3\right)+X_2\omega_2\sin\left(X_3\right)+X_3\omega_1\nonumber\\
&\leq&R(X)+\frac{1}{2}\left(\omega_1^2+\omega_2^2\right).
\end{eqnarray}
Moreover, under (\ref{u1ex}) the dynamics of the nominal, delay-free system are described by $\dot{X}=f\left(X,\kappa_1\left(t,X\right),\omega_2\right)$, where
\begin{eqnarray}
f\left(X,\kappa_1\left(t,X\right),\omega_2\right)=\left[\begin{array}{c}\omega_2\cos\left(X_3\right)\\\omega_2\sin\left(X_3\right)\\-X_3-M\left(X\right)^2\cos\left(t\right)-M\left(X\right)Q\left(X\right)\left(1+\cos^2\left(t\right)\right)\end{array}\right].\label{r1}
\end{eqnarray} 
Therefore, by defining $S_2(X_3)=\frac{1}{2}X_3^2$ we have that
\begin{eqnarray}
\frac{\partial S_2(X_3)}{\partial X_3}\left(-X_3-M(X)^2\cos\left(t\right)-M(X)Q(X)\left(1+\cos^2\left(t\right)\right)\right)\leq 4\left(X_1^4+X_2^4\right),
\end{eqnarray}
where we employed Young's inequality and the fact that $M^2+Q^2=X_1^2+X_2^2$ which follows from (\ref{defr}) and (\ref{defq}). With $S_1(X_1,X_2)=\frac{1}{4}\left(X_1^4+X_2^4\right)$ we get that
\begin{eqnarray}
\frac{\partial S_1(X_1,X_2)}{\partial X_1}\omega_2\cos\left(X_3\right)+\frac{\partial S_1(X_1,X_2)}{\partial X_2}\omega_2\sin\left(X_3\right)\leq 3S_1+\frac{1}{2}\omega_2^4.
\end{eqnarray}
Hence, defining $S=S_1+S_2$ for system $\dot{X}=f\left(X,\kappa_1\left(t,X\right),\omega_2\right)$ with (\ref{r1}) it follows that
\begin{eqnarray}
\frac{\partial S(X)}{\partial X}f\left(X,\kappa_1\left(t,X\right),\omega_2\right)&\leq& 19 S_1+\frac{1}{2}\omega_2^4\nonumber\\
&\leq&19 S+\frac{1}{2}\omega_2^4.
\end{eqnarray}
Since the differential inequality for $S\left(X(t)\right)$ along the trajectories of the system is linear in $S\left(X(t)\right)$ it follows that system $\dot{X}=f\left(X,\kappa_1\left(t,X\right),\omega_2\right)$ with (\ref{r1}) is forward complete with respect to $\omega_2$. Therefore, Assumption \ref{ass2} holds\footnote{ It can be shown that $\dot{M}(t)=U_2(t-D_2)-Q(t)U_1(t-D_1)$, $\dot{Q}(t)=M(t)U_1(t-D_1)$, and $\dot{X}_3(t)=U_1(t-D_1)$. Therefore, when $D_1>D_2$ in order for Assumption \ref{ass2} to hold, the system $\dot{Y}=f\left(Y,\kappa_1\left(t,Y\right),\omega_2\right)$ with $f\left(Y,\kappa_1\left(t,Y\right),\omega_2\right)=\left[\begin{array}{c}U_2(Y)-Y_2\omega_2\\Y_1\omega_2\\\omega_2\end{array}\right]$, where $U_2(Y)=-Y_1+Y_2\left(\sin(t)-\cos(t)\right)-Y_1^2Y_2\cos\left(t\right)-Y_1Y_2^2\left(1+\cos^2\left(t\right)\right)-Y_2Y_3$, has to be forward complete with respect to $\omega_2$. However, this might not be the case. Consider, for example, the case in which $D_2=0$, $\omega_2\equiv0$, $Y_2(0)=1$, $Y_3(0)=0$, and hence, $Y_2(t)=1$, for all $t\geq 0$. The state $Y_1$ satisfies $\dot{Y}_1=-Y_1+\sin(t)-\cos(t)-Y_1^2\cos\left(t\right)-Y_1\left(1+\cos^2\left(t\right)\right)$, for all $t\geq0$. We show next that $Y_1$ escapes to infinity before $t=\frac{\pi}{4}$. Choose $Y_1(0)=-\beta<0$, where $\beta$ is sufficiently large such that $Y_1(t)<0$ for all $t\leq\frac{\pi}{4}$. As long as $Y_1(t)<0$ and $t\leq\frac{\pi}{4}$ it holds that $\dot{Y}_1\leq-3Y_1-Y_1^2\frac{1}{2}$. Hence, using the comparison principle we get that $Y_1(t)\leq -\frac{{6e^{-3t}\beta}}{6+\beta\left(e^{-3t}-1\right)}$. Choosing, for example, $\beta=12$ we have that  $Y_1(t)\leq -\frac{{12}}{2-e^{3t}}$. Hence, $Y_1(t)<0$ for all $t\leq\frac{\log(2)}{3}\approx0.23<\frac{\pi}{4}\approx 0.79$. Moreover, $|Y_1|\to\infty$ before $t=\frac{\log(2)}{3}$ (which also implies that $|X|\to\infty$).}. The closed-loop system in the case $D_1=D_2=0$ is uniformly globally asymptotically stable as it is proved in \cite{pomet}, and hence, one can apply, Theorem \ref{theorem2}.

The predictor-feedback control law is given by
\begin{eqnarray}
U_1(t)&=&-M_1(t)^2\cos\left(t+D_1\right)-M_1(t)Q_1(t)\left(1+\cos^2\left(t+D_1\right)\right)-P_{1_{X_3}}(t)\label{u1ex1}\\
U_2(t)&=&-M_2(t)+Q_2(t)\left(\sin(t+D_2)-\cos(t+D_2)\right)+Q_2(t)\nonumber\\
&&\times\left(-M_2(t)^2\cos\left(t+D_2\right)-M_2(t)Q_2(t)\left(1+\cos^2\left(t+D_2\right)\right)-P_{2_{X_3}}(t)\right),\label{u2ex1}
\end{eqnarray}
where
\begin{eqnarray}
M_i(t)&=&P_{i_{X_1}}(t)\cos\left(P_{i_{X_3}}(t)\right)+P_{i_{X_2}}(t)\sin\left(P_{i_{X_3}}(t)\right),\quad i=1,2\label{defr1}\\
Q_i(t)&=&P_{i_{X_1}}(t)\sin\left(P_{i_{X_3}}(t)\right)-P_{i_{X_2}}(t)\cos\left(P_{i_{X_3}}(t)\right),\quad i=1,2\label{defq1}.
\end{eqnarray}
The predictors are given by
\begin{eqnarray}
P_{1_{X_1}}(t)&=&X_1(t)+\int_{t-D_1}^tU_2(\theta-D_{21})\cos\left(P_{1_{X_3}}(\theta)\right)d\theta\label{p1ex}\\
P_{1_{X_2}}(t)&=&X_2(t)+\int_{t-D_1}^tU_2(\theta-D_{21})\sin\left(P_{1_{X_3}}(\theta)\right)d\theta\\
P_{1_{X_3}}(t)&=&X_3(t)+\int_{t-D_1}^tU_1(\theta)d\theta\label{p1ex1}
\end{eqnarray}
and
\begin{eqnarray}
P_{2_{X_1}}(t)&=&P_{1_{X_1}}(t)+\int_{t-D_{21}}^tU_2(\theta)\cos\left(P_{2_{X_3}}(\theta)\right)d\theta\label{p2ex}\\
P_{2_{X_2}}(t)&=&P_{1_{X_2}}(t)+\int_{t-D_{21}}^tU_2(\theta)\sin\left(P_{2_{X_3}}(\theta)\right)d\theta\\
P_{2_{X_3}}(t)&=&P_{1_{X_3}}(t)+\int_{t-D_{21}}^t\kappa_1(\theta+D_2,P_{2_{X_1}}(\theta),P_{2_{X_2}}(\theta),P_{2_{X_3}}(\theta))d\theta,\label{p2ex1}
\end{eqnarray}
where for all $t-D_{21}\leq\theta\leq t$
\begin{eqnarray}
\kappa_1(\theta+D_2,P_{2_{X_1}}(\theta),P_{2_{X_2}}(\theta),P_{2_{X_3}}(\theta))&=&-M_2(\theta)^2\cos\left(\theta+D_2\right)-M_2(\theta)Q_2(\theta)\nonumber\\
&&\times\left(1+\cos^2\left(\theta+D_2\right)\right)-P_{2_{X_3}}(\theta).\label{pl}
\end{eqnarray}
Note that the $D_1$-time units ahead predictors, namely $P_{1_{X_i}}$, $i=1,2,3$, given by (\ref{p1ex})--(\ref{p1ex1}) can be computed explicitly in terms of the history of $U_1(\theta)$ on the interval $\theta\in[t-D_1,t]$, the history of $U_2(\theta)$ on the interval $\theta\in[t-D_2,t-D_{21}]$, and the current states $X_i(t)$, $i=1,2,3$, as
\begin{eqnarray}
P_{1_{X_1}}(t)&=&X_1(t)+\int_{t-D_1}^tU_2(\theta-D_{21})\cos\left(X_3(t)+\int_{t-D_1}^{\theta}U_1(s)ds\right)d\theta\label{p1exex}\\
P_{1_{X_2}}(t)&=&X_2(t)+\int_{t-D_1}^tU_2(\theta-D_{21})\sin\left(X_3(t)+\int_{t-D_1}^{\theta}U_1(s)ds\right)d\theta\\
P_{1_{X_3}}(t)&=&X_3(t)+\int_{t-D_1}^tU_1(\theta)d\theta\label{p1ex1ex}.
\end{eqnarray}
This is not possible for the $D_2$-time units ahead predictors, namely $P_{2_{X_i}}$, $i=1,2,3$, given by (\ref{p2ex})--(\ref{pl}), since in this case $P_{2_{X_3}}$ defined in (\ref{p2ex1}) can not be solved explicitly in terms of the current states of the plant and the history of the actuator states.




\section{Application to Linear Systems}
\label{linear section}
We specialize our control design to the case of linear systems in which case the control laws are given explicitly. We consider the following system
\begin{eqnarray}
\dot{X}(t)=AX(t)+\sum_{i=1}^{i=m}b_iU_i(t-D_i),\label{kl}
\end{eqnarray}
which can be written as 
\begin{eqnarray}
\dot{X}(t)&=&AX(t)+\sum_{i=1}^{i=m}b_iu_i(0,t)\label{plant1lin}\\
\partial_t u_i(x,t)&=&\partial_x u_i(x,t),\quad x\in(0,D_i),\quad i=1,\ldots,m\label{pde1lin}\\
u_i\left(D_i,t\right)&=&U_i(t),\quad i=1,\ldots,m\label{pde2lin},
\end{eqnarray}
where $A$ is an $n\times n$ matrix and $b_i$, $i=1,\ldots,m$ are $n$-dimensional vectors. System (\ref{plant1lin})--(\ref{pde2lin}) is the linear version of system (\ref{plant1})--(\ref{pde2}). We assume linear nominal feedback control laws, that is, in the delay-free case we have $U_i(t)=k_i^TX(t)$, and hence, Assumption \ref{ass3} is satisfied with $\kappa_i(X)=k_i^TX$, $i=1,\ldots,m$, under the assumption that the pair $\left(A,\left[\begin{array}{ccc}b_1&\ldots&b_m\end{array}\right]\right)$ is stabilizable. Note that Assumptions \ref{ass1} and \ref{ass2} hold for the case of linear systems under linear nominal feedback controllers. We first re-write the predictor states in their PDE representation, namely (\ref{p1ode})--(\ref{pmb}), for the case of linear systems given by (\ref{plant1lin})--(\ref{pde2lin}) as
\begin{eqnarray}
\partial_xp_1(x,t)&=&Ap_1(x,t)+\sum_{i=1}^{i=m}b_iu_i(x,t),\quad x\in[0,D_1]\label{p1odelin}\\
\partial_xp_2(x,t)&=&\left(A+b_1k_1^T\right)p_2(x,t)+\sum_{i=2}^{i=m}b_iu_i(x,t),\quad x\in[D_1,D_2]\label{p2odelin}\\
&\vdots&\nonumber\\
\partial_xp_m(x,t)&=&\left(A+\sum_{i=1}^{i=m-1}b_ik_i^T\right)p_m(x,t)+b_mu_m(x,t),\quad x\in[D_{m-1},D_m],\label{pmodelin}
\end{eqnarray}
with initial conditions given by (\ref{p1b})--(\ref{pmb}). Solving explicitly the linear ODEs in $x$ (\ref{p1odelin})--(\ref{pmodelin}) and using the boundary conditions  (\ref{p1b})--(\ref{pmb}) we get that
\begin{eqnarray}
p_1(x,t)&=&e^{Ax}X(t)+\int_0^xe^{A(x-y)}\sum_{i=1}^{i=m}b_iu_i(y,t)dy,\quad x\in[0,D_1]\label{p1odelinexp}\\
p_2(x,t)&=&e^{A_1(x-D_1)}p_1(D_1,t)+\int_{D_1}^xe^{A_1(x-y)}\sum_{i=2}^{i=m}b_iu_i(y,t)dy,\quad x\in[D_1,D_2]\label{p2odelinexp}\\
&\vdots&\nonumber\\
p_m(x,t)&=&e^{A_{m-1}(x-D_{m-1})}p_{m-1}(D_{m-1},t)+\int_{D_{m-1}}^xe^{A_{m-1}(x-y)}b_mu_m(y,t)dy,\nonumber\\
&& x\in[D_{m-1},D_m],\label{pmodelinexp}
\end{eqnarray}
where we used the notation
\begin{eqnarray}
A_0&=&A\label{a0}\\
A_i&=&A+\sum_{j=1}^{j=i}b_jk_j^T,\quad i=1,\ldots,m.\label{ai}
\end{eqnarray}
The control laws are given by 
\begin{eqnarray}
U_i(t)&=&\kappa_i^Tp_i(D_i,t),\quad i=1,\ldots,m\label{connew1lin}.
\end{eqnarray}

\begin{theorem}
\label{theoremlin}
Consider the closed-loop system consisting of the plant (\ref{plant1lin})--(\ref{pde2lin}) and the control laws (\ref{connew1lin}), (\ref{p1odelinexp})--(\ref{pmodelinexp}). Let the pair $\left(A,\left[\begin{array}{ccc}b_1&\ldots&b_m\end{array}\right]\right)$ be stabilizable. There exist positive constants $\lambda$ and $\mu$ such that for all initial conditions $X_0\in\mathbb{R}^n$ and $u_{i_0}\in H^1(0,D_i)$, $i=1,\ldots,m$, which are compatible with the feedback laws, the closed-loop system has a unique solution $\left(X(t),u_1(\cdot,t),\ldots,u_m(\cdot,t)\right)\in C\left([0,\infty);\mathbb{R}^n\times H^1(0,D_1)\times\ldots\times H^1(0,D_m)\right)\cap C^1\left([0,\infty);\mathbb{R}^n\times L^2(0,D_1)\times\ldots\times L^2(0,D_m)\right)$ and the following holds
\begin{eqnarray}
\Gamma(t)\leq\mu \Gamma(0) e^{-\lambda t}, \quad\mbox{for all $t\geq0$},\label{estimate theorem 1lin}
\end{eqnarray}
where
\begin{eqnarray}
\Gamma(t)=|X(t)|+\sum_{i=1}^{i=m}\int_0^{D_i}u_i(x,t)^2dx.\label{normlin}
\end{eqnarray}
\end{theorem}

Note that with definitions $p_i(x,t)=P_i(t+x-D_i)$, for all $x\in [D_{i-1},D_i]$, with $D_0=0$ (see Section \ref{sudpre}) the predictors (\ref{p1odelinexp})--(\ref{pmodelinexp}) can be written as 
\begin{eqnarray}
P_1(\theta)&=&e^{A(\theta-t+D_1)}X(t)+\int_{t-D_1}^{\theta}e^{A(\theta-s)}\sum_{i=1}^{i=m}b_iU_i(s-D_{i1})ds,\quad t-D_1\leq\theta\leq t\label{P11inlin}\\
P_2(\theta)&=&e^{A_1(\theta-t+D_{21})}P_1(t)+\int_{t-D_{21}}^{\theta}e^{A_1(\theta-s)}\sum_{i=2}^{i=m}b_iU_i(s-D_{i2})ds,\quad t-D_{21}\leq\theta\leq t\label{conninlin}\\
&\vdots&\nonumber\\
P_m(\theta)&\!\!\!\!=\!\!\!\!&e^{A_{m-1}(\theta-t+D_{mm-1})}P_{m-1}(t)+\int_{t-D_{mm-1}}^{\theta}\!e^{A_{m-1}(\theta-s)}b_mU_m(s)ds,\quad\!\! t-D_{mm-1}\leq\theta\leq t\label{conn1inlin},
\end{eqnarray}
where the matrices $A_i$, $i=1,\ldots,m-1$, are given in (\ref{a0}), (\ref{ai}). With this notation the control laws (\ref{connew1lin}) are written as
\begin{eqnarray}
U_i(t)&=&\kappa_i^TP_i(t),\quad i=1,\ldots,m\label{connew1lindelay}.
\end{eqnarray}
We have the following corollary as an immediate consequence of Theorem and relation (\ref{solu}). 

\begin{corollary}[The version of Theorem \ref{theoremlin} in standard delay notation]
\label{cor1n}
Consider the closed-loop system consisting of the plant (\ref{kl}) and the control laws (\ref{connew1lindelay}), (\ref{P11inlin})--(\ref{conn1inlin}). Under the assumption that the pair $\left(A,\left[\begin{array}{ccc}b_1&\ldots&b_m\end{array}\right]\right)$ is stabilizable the following holds
\begin{eqnarray}
\left|X(t)\right|+\sum_{i=1}^{i=m}\int_{t-D_i}^{t}U_i(\theta)^2d\theta\leq\mu\left(\left|X(0)\right|+\sum_{i=1}^{i=m}\int_{-D_i}^0U_i(\theta)^2d\theta\right)e^{-\lambda t}, \quad\mbox{for all $t\geq0$}.\label{estimate cor 1n}
\end{eqnarray}
\end{corollary}

\begin{proof}
We prove Theorem \ref{theoremlin} by showing that the control laws (\ref{connew1lin}), (\ref{p1odelinexp})--(\ref{pmodelinexp}) are identical to the ones introduced in \cite{daisuke new} and \cite{daisuke}, and hence, Theorem \ref{theoremlin} is proved by following the proof of Theorem 1 in \cite{daisuke}. Equivalently we show that the backstepping transformations (\ref{leminvz}) specialized to the linear case are identical to the backstepping transformations introduced in \cite{daisuke}. Toward that end define
\begin{eqnarray}
\phi_i(x)&=&\left\{\begin{array}{ll}x,&x\leq D_i\\D_i,&x\geq D_{i}\end{array}\right.,\quad i=1,\ldots,m\label{defphi}\\
g_i(x,t)&=&\left\{\begin{array}{ll}p_1(x,t),&x\in[0,D_1]\\p_2(x,t),&x\in[D_1,D_2]\\\vdots\\p_i(x,t),&x\in[D_{i-1},D_i]\end{array}\right.,\quad i=1,\ldots,m.\label{go}
\end{eqnarray}
Then, using (\ref{p1odelinexp})--(\ref{pmodelinexp}) it follows that
\begin{eqnarray}
e^{-A_{i-1}x}g_i(x,t)=v_{i-1}(x,t),\quad x\in[0,D_{i}],\quad i=1,\ldots,m,\label{com}
\end{eqnarray}
where for all $x\in[0,D_m]$
\begin{eqnarray}
v_i(x,t)&=&e^{-A_i\phi_i(x)}e^{A_{i-1}\phi_i(x)}v_{i-1}(x,t)+\int_{\phi_i(x)}^{\phi_{i+1}(x)}e^{-A_iy}\sum_{j=i+1}^{j=m}b_ju_j(y,t)dy\label{91}\\
v_0(x,t)&=&X(t)+\int_0^{\phi_1(x)}e^{-Ay}\sum_{j=1}^{j=m}b_ju_j(y,t)dy.\label{92}
\end{eqnarray}
We show this by induction. For all $x\in[0,D_1]$ it holds that 
\begin{eqnarray}
e^{-Ax}g_1(x,t)&=&e^{-Ax}p_1(x,t)\nonumber\\
&=&X(t)+\int_0^{\phi_1(x)}e^{-Ay}\sum_{j=1}^{j=m}b_ju_j(y,t)dy\nonumber\\
&=&v_0(x),
\end{eqnarray}
where we used (\ref{p1odelinexp}) and (\ref{defphi}). Assume now that (\ref{com}) holds for some $i$. It follows that
\begin{eqnarray}
e^{-A_ix}g_{i+1}(x,t)&=&e^{-A_ix}\left\{\begin{array}{ll}p_1(x,t),&x\in[0,D_1]\\p_2(x,t),&x\in[D_1,D_2]\\\vdots\\p_i(x,t),&x\in[D_{i-1},D_i]\\p_{i+1}(x,t),&x\in[D_i,D_{i+1}]\end{array}\right.\nonumber\\
&=&e^{-A_ix}\left\{\begin{array}{ll}e^{A_{i-1}x}v_{i-1}(x,t),&x\in[0,D_i]\\p_{i+1}(x,t),&x\in[D_i,D_{i+1}]\end{array}\right.,\label{f1}
\end{eqnarray}
for all $x\in[0,D_{i+1}]$. Using (\ref{p1odelinexp})--(\ref{pmodelinexp}) and (\ref{defphi}) we get from (\ref{f1}) that
\begin{eqnarray}
e^{-A_ix}g_{i+1}(x,t)&=&\left\{\begin{array}{ll}e^{-A_i\phi_i(x)}e^{A_{i-1}\phi_i(x)}v_{i-1}(x,t),&x\in[0,D_i]\\e^{-A_iD_{i}}p_{i}(D_{i},t)+\int_{D_{i}}^xe^{-A_{i}y}\sum_{j=i+1}^{j=m}b_ju_j(y,t)dy,&x\in[D_i,D_{i+1}]\end{array}\right..\label{f2}
\end{eqnarray}
Using (\ref{go}) and (\ref{com}) it follows that 
\begin{eqnarray}
e^{-A_{i-1}D_i}g_i(D_i,t)&=&e^{-A_{i-1}D_i}p_i(D_i,t)\nonumber\\
&=&v_{i-1}(D_i,t),
\end{eqnarray}
and hence,
\begin{eqnarray}
p_i(D_i,t)=e^{A_{i-1}D_i}v_{i-1}(D_i,t).\label{sh}
\end{eqnarray}
Combining (\ref{f2}) with (\ref{sh}) we arrive at
\begin{eqnarray}
e^{-A_ix}g_{i+1}(x,t)&=&\left\{\begin{array}{ll}e^{-A_i\phi_i(x)}e^{A_{i-1}\phi_i(x)}v_{i-1}(x,t),&x\in[0,D_i]\\e^{-A_iD_{i}}e^{A_{i-1}D_i}v_{i-1}(D_{i},t)+\int_{D_{i}}^xe^{-A_{i}y}\sum_{j=i+1}^{j=m}b_ju_j(y,t)dy,&x\in[D_i,D_{i+1}]\end{array}\right..\label{f3}
\end{eqnarray}
We then observe from (\ref{defphi}) that $\phi_i(x)=\phi_{i+1}(x)=x$, for all $x\leq D_i$, and hence,
\begin{eqnarray}
\int_{\phi_i(x)}^{\phi_{i+1}(x)}e^{-A_iy}\sum_{j=i+1}^{j=m}b_ju_j(y,t)dy=0,\quad x\in[0,D_i].
\end{eqnarray}
Moreover, using (\ref{defphi}) we get that 
\begin{eqnarray}
e^{-A_iD_{i}}e^{A_{i-1}D_i}=e^{-A_i\phi_i(x)}e^{A_{i-1}\phi_i(x)},\quad x\in[D_i,D_{i+1}].
\end{eqnarray}
The proof is completed if we show that 
\begin{eqnarray}
v_{i-1}(D_{i},t)=v_{i-1}(x,t),\quad \mbox{for all $x\in[D_{i},D_m]$ and $i=1,\ldots,m$}. \label{gh}
\end{eqnarray}
We prove this by induction and by using definitions (\ref{91}), (\ref{92}). Using (\ref{defphi}) for $i=1$ we have that
\begin{eqnarray}
v_0(D_1,t)&=&X(t)+\int_0^{D_1}e^{-Ay}\sum_{j=1}^{j=m}b_ju_j(y,t)dy\nonumber\\
&=&X(t)+\int_0^{\phi_1(x)}e^{-Ay}\sum_{j=1}^{j=m}b_ju_j(y,t)dy,\quad \mbox{for all $x\geq D_1$}\nonumber\\
&=&v_0(x,t),\quad \mbox{for all $x\geq D_1$}.
\end{eqnarray}
Assume now that (\ref{gh}) holds for some $i$. Then using (\ref{defphi}) and (\ref{91}) we have that
\begin{eqnarray}
v_i(D_{i+1},t)&=&e^{-A_iD_i}e^{A_{i-1}D_{i}}v_{i-1}(D_{i+1},t)+\int_{D_i}^{D_{i+1}}e^{-A_iy}\sum_{j=i+1}^{j=m}b_ju_j(y,t)dy\nonumber\\
&=&e^{-A_i\phi_i(x)}e^{A_{i-1}\phi_i(x)}v_{i-1}(D_{i+1},t)+\int_{\phi_i(x)}^{\phi_{i+1}(x)}e^{-A_iy}\sum_{j=i+1}^{j=m}b_ju_j(y,t)dy,\nonumber\\
&& \mbox{for all $x\geq D_{i+1}$}.
\end{eqnarray}
Therefore, using (\ref{gh}) we get that
\begin{eqnarray}
v_i(D_{i+1},t)&=&e^{-A_i\phi_i(x)}e^{A_{i-1}\phi_i(x)}v_{i-1}(x,t)+\int_{\phi_i(x)}^{\phi_{i+1}(x)}e^{-A_iy}\sum_{j=i+1}^{j=m}b_ju_j(y,t)dy\nonumber\\
&=&v_i(x,t),\quad \mbox{for all $x\geq D_{i+1}$},
\end{eqnarray}
which completes the proof that (\ref{com}) holds. Therefore, using (\ref{go}) the backstepping transformations (\ref{eq2})--(\ref{bacmnew}) can be written in the linear case as
 \begin{eqnarray}
w_1(x,t)&=&u_1(x,t)-\kappa_1^Te^{A_0x}v_0(x,t),\quad x\in[0,D_1]\label{eq2wew}\\
w_2(x,t)&=&u_2(x,t)-\kappa_2^Te^{A_1x}v_1(x,t),\quad x\in[0,D_2]\label{bac2newwew}\\
&\vdots&\nonumber\\
w_m(x,t)&=&u_m(x,t)-\kappa_m^Te^{A_{m-1}x}v_{m-1}(x,t),\quad x\in[0,D_m],\label{bacmnewwew}
\end{eqnarray}
which are identical to the backstepping transformations (67), (68) from \cite{daisuke}.
\end{proof}

As a special case we provide explicitly the first two backstepping transformations. Using (\ref{91}), (\ref{92}) the backstepping transformations (\ref{eq2wew}), (\ref{bac2newwew}) take the form
 \begin{eqnarray}
w_1(x,t)\!\!&=&\!\!u_1(x,t)-k_1^T\left(e^{Ax}X(t)+\int_0^{x}e^{A(x-y)}\sum_{i=1}^{i=m}b_iu_i(y,t)dy\right),\quad x\in[0,D_1]\label{eq21}\\
w_2(x,t)\!\!&=&\!\!u_2(x,t)-\left\{\begin{array}{ll}k_2^T\left(e^{Ax}X(t)+\int_0^{x}e^{A(x-y)}\sum_{i=1}^{i=m}b_iu_i(y,t)dydy\right),\quad x\in[0,D_1]\\k_2^T\left(e^{A_1(x-D_1)}e^{AD_1}X(t)+e^{A_1(x-D_1)}\int_0^{D_1}e^{A(D_1-y)}\sum_{i=1}^{i=m}b_iu_i(y,t)dy\right.\\+\left.\int_{D_1}^xe^{A_1(x-y)}\sum_{i=2}^{i=m}b_iu_i(y,t)dy\right),\quad x\in[D_1,D_2]\end{array}\right.\!\!,\label{bac2new1}
\end{eqnarray}
which are identical to relations (15), (77), and (78) from \cite{daisuke}. The predictor feedback control laws are
\begin{eqnarray}
u_1(D_1,t)&=&k_1^T\left(e^{AD_1}X(t)+\int_0^{D_1}e^{A(D_1-y)}\left(b_1u_1(y,t)+b_2u_2(y,t)\right)dy\right)\\
u_2(D_2,t)&=&k_2^T\left(e^{A_1D_{21}}e^{AD_1}X(t)+e^{A_1D_{21}}\int_0^{D_1}e^{A(D_1-y)}b_1u_1(y,t)dy\right.\nonumber\\
&&\left.\vphantom{\int_{D_1}^{D_2}}+e^{A_1D_{21}}\int_0^{D_1}e^{A(D_1-y)}b_2u_2(y,t)dy+\int_{D_1}^{D_2}e^{A_1(D_2-y)}b_2u_2(y,t)dy\right),
\end{eqnarray}
and in terms of the delayed states $U_i(t+x-D_i)=u_i(x,t)$ by
\begin{eqnarray}
U_1(t)&=&k_1^T\left(e^{AD_1}X(t)+\int_{t-D_1}^{t}e^{A(t-\theta)}b_1U_1(\theta)d\theta+\int_{t-D_2}^{t-D_{21}}e^{A(t-\theta-D_{21})}b_2U_2(\theta)d\theta\right)\label{comcon1}\\
U_2(t)&=&k_2^T\left(e^{A_1D_{21}}e^{AD_1}X(t)+e^{A_1D_{21}}\int_{t-D_1}^{t}e^{A(t-\theta)}b_1U_1(\theta)d\theta\right.\nonumber\\
&&\left.\vphantom{\int_{D_1}^{D_2}}+e^{A_1D_{21}}\int_{t-D_2}^{t-D_{21}}e^{A(t-\theta-D_{21})}b_2U_2(\theta)d\theta+\int_{t-D_{21}}^{t}e^{A_1(t-\theta)}b_2U_2(\theta)d\theta\right).\label{comcon2}
\end{eqnarray}
Relations (\ref{comcon1}), (\ref{comcon2}) can be written as
\begin{eqnarray}
U_1(t)&=&k_1^TP_1(t)\label{comcon11}\\
U_2(t)&=&k_2^T\left(e^{A_1D_{21}}P_1(t)+\int_{t-D_{21}}^{t}e^{A_1(t-\theta)}b_2U_2(\theta)d\theta\right),\label{comcon21}
\end{eqnarray}
which are the control laws (46), (47) from \cite{daisuke}. Since the explicit expression of the general $i$-th control law for the general case of $m$ inputs is identical to the one obtained in \cite{daisuke} (relation (43)), for clarity of exposition we provide it again here
\begin{eqnarray}
U_i(t)&=&k_i^T\left(\vphantom{\sum_{j=1}^{j=i}\int_{t-D_j}^t}\Phi(D_i,0)X(t)+\sum_{j=1}^{j=i}\int_{t-D_j}^t\Phi\left(D_i,\theta-t+D_j\right)b_jU_j(\theta)d\theta\right.\nonumber\\
&&+\left.\sum_{j=i+1}^{j=m}\int_{t-D_j}^{t-D_{ji}}\Phi\left(D_i,\theta-t+D_j\right)b_jU_j(\theta)d\theta\right),
\end{eqnarray}
where
\begin{eqnarray}
\Phi(x,y)=e^{A_i(x-D_i)}e^{A_{i-1}(D_i-D_{i-1})}\ldots e^{A_j(D_{j+1}-y)},\quad \mbox{for all $D_i\leq x\leq D_{i+1}$ and $D_j\leq y\leq\phi_{j+1}(x)$},
\end{eqnarray}
for any $i,j\in\left\{0,1,\ldots, m-1\right\}$ satisfying $0\leq i<j$, and
\begin{eqnarray}
\Phi(x,y)=e^{A_i(x-y)},\quad \mbox{for all $D_i\leq x\leq y\leq D_{i+1}$},
\end{eqnarray}
for $i\in\left\{0,1,\ldots, m-1\right\}$, where $\phi_i$ is defined in (\ref{defphi}) and we adopt the notation $D_0=0$.
\section{Simulations}
\label{sec sim}
We consider the stabilization problem of the nonholonomic unicycle subject to distinct input delays from Section \ref{uni} whose dynamics are described by (\ref{sysex1})--(\ref{sysexn}). We choose $D_1=0.5$ and $D_2=1$. The initial conditions for the plant are chosen as $X_1(0)=X_2(0)=X_3(0)=0.5$ and for the actuator states as $U_1(\theta)=0$, for all $-D_1\leq\theta\leq0$, and $U_2(\theta)=0$, for all $-D_2\leq\theta\leq0$. In Fig. \ref{fig1} we show the response of the closed-loop system under the predictor-feedback controller (\ref{u1ex1})--(\ref{pl}) and in Fig. \ref{fig2} the corresponding control efforts. The predictor-feedback controller asymptotically stabilizes the nonholonomic unicycle. In particular, at $t=1$ both delays are compensated and the system behaves as in the nominal, delay-free case. In contrast, at is is shown in Fig. \ref{fig3}, the closed-loop system under the nominal, uncompensated controller is unstable.
\begin{figure}
\centering
\includegraphics[width=0.733\linewidth]{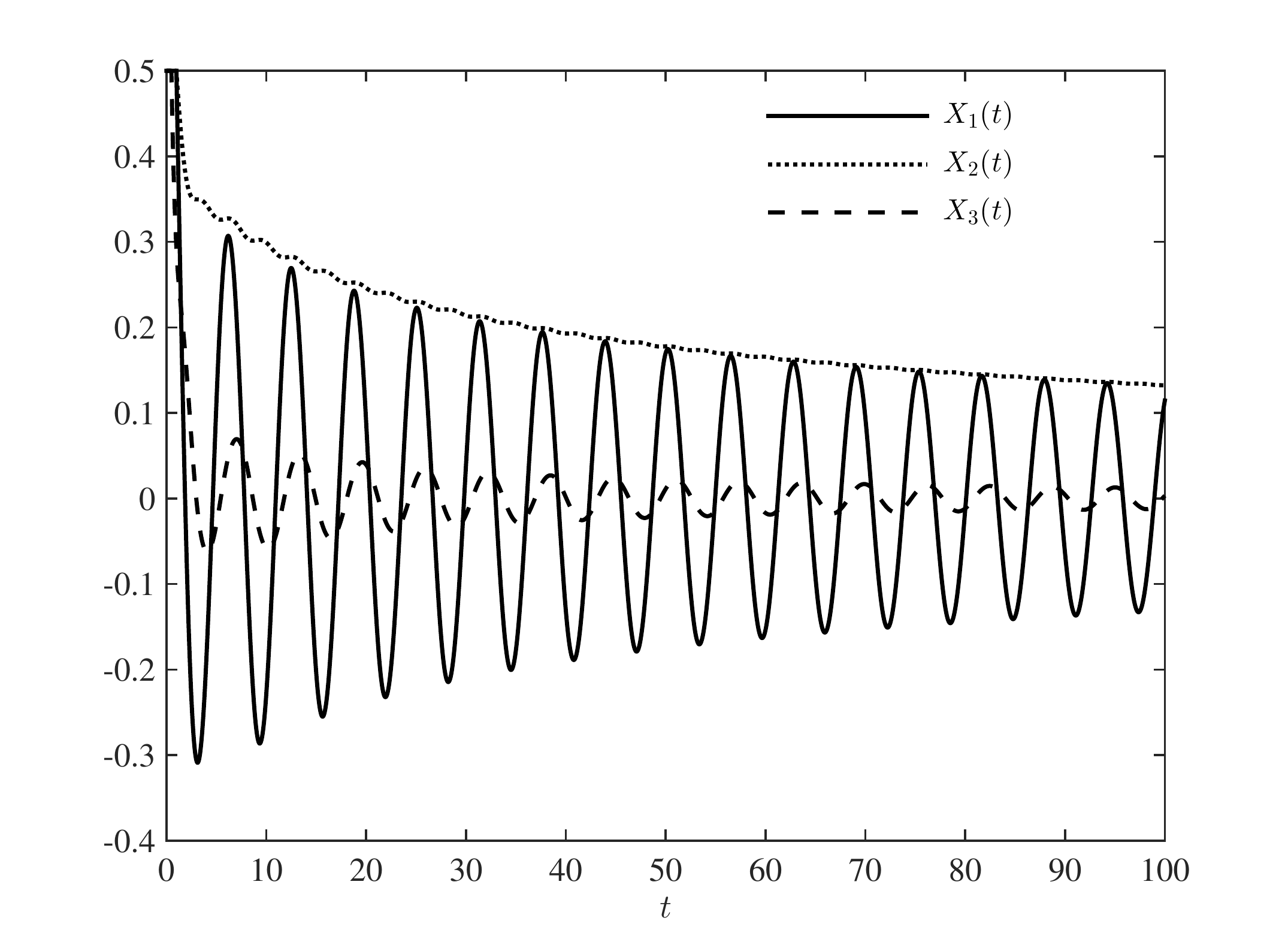}
\caption{The closed-loop response of system (\ref{sysex1})--(\ref{sysexn}) with $D_1=0.5$ and $D_2=1$ under the predictor-feedback controller (\ref{u1ex1})--(\ref{pl}). The initial conditions of the plant are $X_1(0)=X_2(0)=X_3(0)=0.5$ and for the actuator states are $U_1(\theta)=0$, for all $-0.5\leq\theta\leq0$, and $U_2(\theta)=0$, for all $-1\leq\theta\leq0$.}
\label{fig1}
\end{figure}

\begin{figure}
\centering
\includegraphics[width=0.733\linewidth]{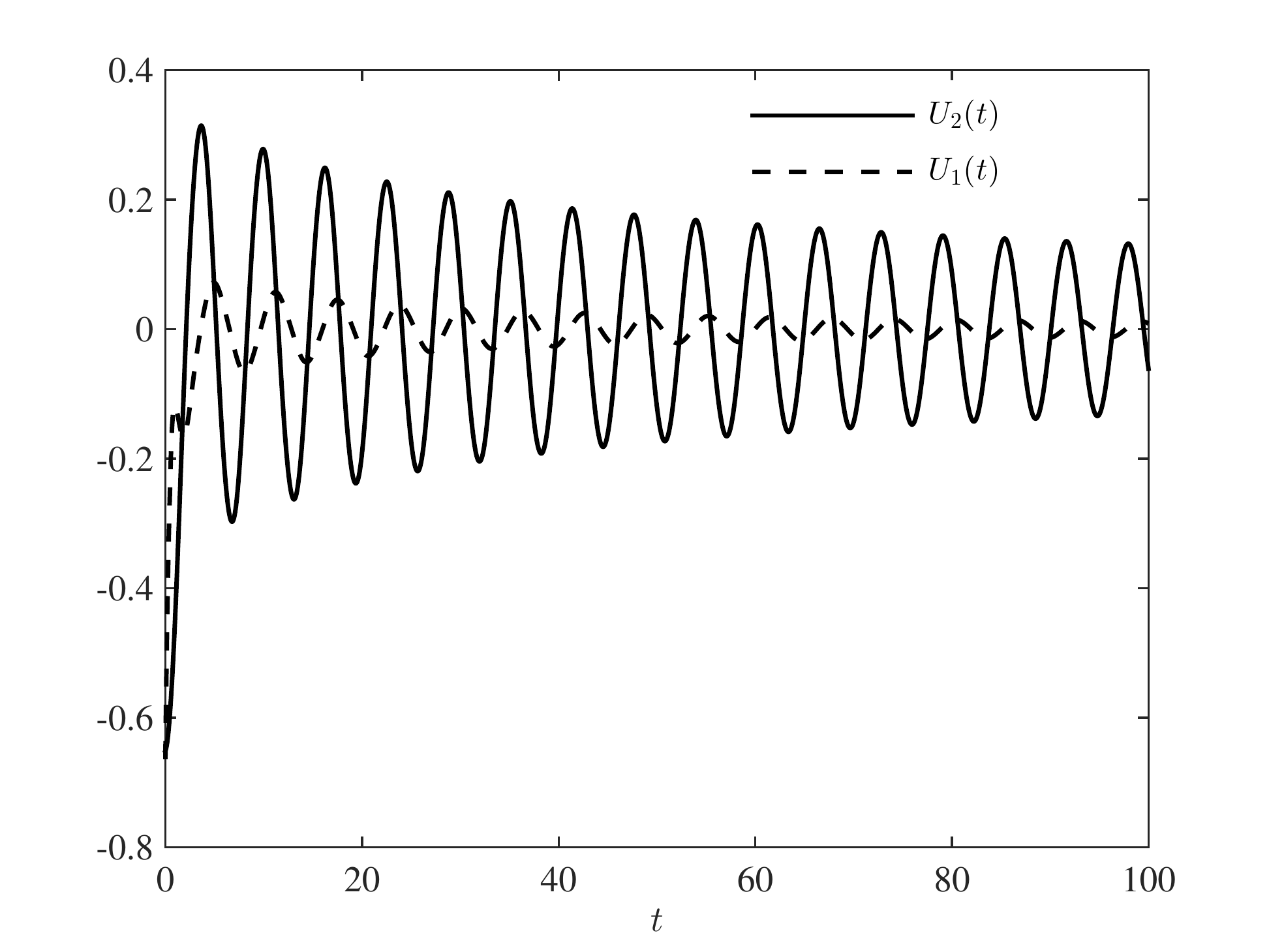}
\caption{The control efforts (\ref{u1ex1}) and (\ref{u2ex1}) of the closed-loop response of system (\ref{sysex1})--(\ref{sysexn}) with $D_1=0.5$ and $D_2=1$ under the predictor-feedback controller (\ref{u1ex1})--(\ref{pl}). The initial conditions of the plant are $X_1(0)=X_2(0)=X_3(0)=0.5$ and for the actuator states are $U_1(\theta)=0$, for all $-0.5\leq\theta\leq0$, and $U_2(\theta)=0$, for all $-1\leq\theta\leq0$.}
\label{fig2}
\end{figure}

\begin{figure}
\centering
\includegraphics[width=0.733\linewidth]{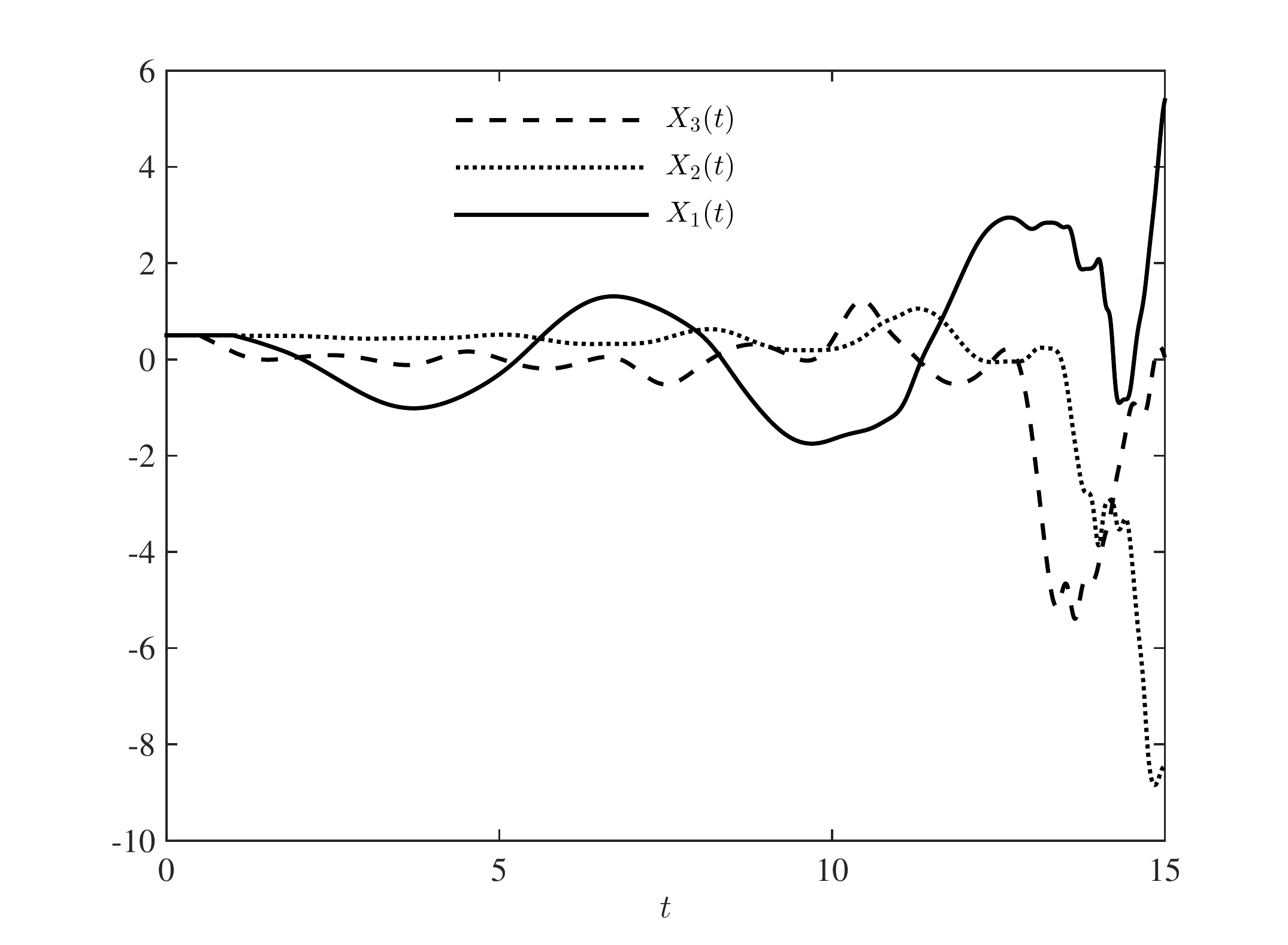}
\caption{The closed-loop response of system (\ref{sysex1})--(\ref{sysexn}) with $D_1=0.5$ and $D_2=1$ under the nominal controller (\ref{u1ex})--(\ref{defq}). The initial conditions of the plant are $X_1(0)=X_2(0)=X_3(0)=0.5$ and for the actuator states are $U_1(\theta)=0$, for all $-0.5\leq\theta\leq0$, and $U_2(\theta)=0$, for all $-1\leq\theta\leq0$.}
\label{fig3}
\end{figure}

\section{Conclusions}
We presented a predictor-feedback control design methodology for the stabilization of multi-input nonlinear systems with distinct input delays. We proved global asymptotic stability of the closed-loop system using Lyapunov arguments and arguments based on estimates of closed-loop solutions. We also dealt with linear systems as a special case of our methodology. We applied our approach to the stabilization of the nonholonomic unicycle with delayed inputs.


In contrast to the single-input case, for which the predictor-feedback controller is available explicitly for some classes of nonlinear systems (see, for instance, \cite{kar2}) with a specific open-loop structure, in the multi-input case the class of nonlinear systems for which the predictor-feedback control law can be obtained explicitly seems to be more restrictive (although it can be obtained explicitly in some trivial cases, such as, for example, the case of linear systems). This is attributed to the fact that the formulae of the predictors in the multi-input case depends not only on the open-loop structure of the system but also on the form of the feedback functions.


As a next step, it is of interest to study the problem of stabilization of multi-input nonlinear systems with actuator dynamics governed by wave or diffusion PDEs with different wave propagation speeds or diffusion coefficients, respectively, in each individual input channel. The starting point for such a study is \cite{bek8}. 

\setcounter{equation}{0}
\renewcommand{\theequation}{A.\arabic{equation}}
\appendices
\section*{Appendix A}

\subsection*{Proof of Lemma \ref{leminvz}}
\label{proof leminvz}
Setting $x=0$ into (\ref{eq2})--(\ref{bacmnew}), and (\ref{p1}) we get (\ref{plant1tr}). Using (\ref{sol1}) one can conclude that $\partial_tp_1(x,t)=\partial_xp_1(x,t)$, for all $x\in[0,D_1]$. Hence, using (\ref{pde1}) and (\ref{eq2}) we get that
\begin{eqnarray}
\partial_tw_1(x,t)-\partial_xw_1(x,t)&=&\partial_tu_1(x,t)-\partial_xu_1(x,t)+\frac{\partial \kappa_1\left(p_1(x,t)\right)}{\partial p}\left(\partial_tp_1(x,t)-\partial_xp_1(x,t)\right)\nonumber\\
&=&0.
\end{eqnarray}
Similarly, using the fact that $p_i(x,t)=X(t+x)$, for all $x\in[D_{i-1},D_i]$ and $i=2,\ldots,m$, we get (\ref{pde1tr}). Setting $x=D_1$ into (\ref{eq2})--(\ref{bacmnew}) and using (\ref{pde2}), (\ref{connew1}) we arrive at (\ref{pde2tr}). 

\subsection*{Proof of Lemma \ref{inversezeta}}
\label{proof inversezeta}
We prove this lemma by showing that $p_1(x,t)=\pi_1(x,t)$, for all $x\in[0,D_1]$, and $p_i(x,t)=\pi_i(x,t)$, for all $x\in[D_{i-1},D_i]$, $i=2,\ldots,m$. Equivalently we show that $\pi_i(x,t)=X(t+x)$, $i=1,\ldots,m$ and use the fact that $p_i(x,t)=X(t+x)$, $i=1,\ldots,m$. We first observe that $\pi_1(0,t)=X(t)$, and hence, $\pi_1$ satisfies the following initial value problem for all $x\in[0,D_1]$
\begin{eqnarray}
\partial_x\pi_1(x,t)&=&f\left(\pi_1(x,t),w_1(x,t)+\kappa_1\left(\pi_1(x,t)\right),\ldots,w_m(x,t)+\kappa_m\left(\pi_1(x,t)\right)\right)\label{pode1inv1}\\
\pi_1(0,t)&=&X(t).\label{pode2inv}
\end{eqnarray}
The solution to (\ref{pode1inv1}), (\ref{pode2inv}) is $\pi_1(x,t)=X(t+x)$. This solution satisfies the boundary condition (\ref{pode2inv}). It also satisfies the ODE (\ref{pode1inv1}) since by (\ref{plant1tr}), (\ref{pde1tr}) it follows that the following holds
\begin{eqnarray}
X'(t+x)&=&f\left(X(t+x),W_1(t+x-D_1)+\kappa_1\left(X(t+x)\right),\right.\nonumber\\
&&\left.\ldots,W_m(t+x-D_m)+\kappa_m\left(X(t+x)\right)\right),\quad \mbox{for all $t\geq0$ and $0\leq x\leq D_1$},
\end{eqnarray} 
where the solutions to (\ref{pde1tr}), (\ref{pde2tr}) are given by $w_i(x,t)=W(t+x-D_i)$, for all $x\in[0,D_i]$ and $i=1,\ldots,m$, where $W_i$, $i=1,\ldots,m$, satisfy $W_i(t)=0$, for all $t\geq0$.
Assume now that $\pi_j(x,t)=X(t+x)$, for all $x\in[D_{j-1},D_j]$, and some $j$. Then we claim that $\pi_{j+1}=X(t+x)$, for all $x\in [D_j,D_{j+1}]$. The function $\pi_{j+1}$ satisfies the following initial value problem
\begin{eqnarray}
\partial_x\pi_{j+1}(x,t)&=&f\left(\pi_{j+1}(x,t),\kappa_1\left(\pi_{j+1}(x,t)\right),\ldots,\kappa_j\left(\pi_{j+1}(x,t)\right),w_{j+1}(x,t)\right.\nonumber\\
&&\left.+\kappa_{j+1}\left(\pi_{j+1}(x,t)\right),\ldots,w_{m}(x,t)+\kappa_{m}\left(\pi_{j+1}(x,t)\right)\right),\quad x\in [D_j,D_{j+1}]\label{eqn1}\\
\pi_{j+1}(D_j,t)&=&\pi_j(D_j,t).\label{eqn1b}
\end{eqnarray}
Since $\pi_j(x,t)=X(t+x)$, for all $x\in[D_{j-1},D_j]$, it follows that $\pi_{j+1}(x,t)=X(t+x)$, $x\in[D_j,D_{j+1}]$, satisfies the boundary condition (\ref{eqn1b}). Using (\ref{pde1tr}), (\ref{pde2tr}) we get that $w_i(x,t)=W_i(t+x-D_i)$, for all $x\in [0,D_i]$, where $W_i(t)=0$, for all $t\geq0$. Since from (\ref{plant1tr}) it holds that 
\begin{eqnarray}
X'(t+x)&=&f\left(X(t+x),W_1(t+x-D_1)+\kappa_1\left(X(t+x)\right),\right.\nonumber\\
&&\left.\ldots,W_m(t+x-D_m)+\kappa_m\left(X(t+x)\right)\right),\mbox{for all $t\geq0$ and $x\geq0$}\label{plant1trex},
\end{eqnarray}
one can conclude that the following hods for all $t\geq0$ and $x\in [D_j,D_{j+1}]$
\begin{eqnarray}
X'(t+x)&=&f\left(X(t+x),\kappa_1\left(X(t+x)\right),\ldots,\kappa_j\left(X(t+x)\right),W_{j+1}(t+x-D_{j+1})\right.\nonumber\\
&&\left.+\kappa_{j+1}\left(X(t+x)\right),\ldots,W_m(t+x-D_m)+\kappa_m\left(X(t+x)\right)\right)\label{plant1trex1},
\end{eqnarray}
and hence, $\pi_{j+1}=X(t+x)$, for all $x\in [D_j,D_{j+1}]$. Since this holds for an arbitrary $j$ one can conclude that $\pi_{j}=X(t+x)$, for all $x\in [D_{j-1},D_{j}]$  and $j=1,\ldots,m$, with $D_0=0$, which completes the proof.

\subsection*{Proof of Lemma \ref{lemcl}}
\label{proof lemcl}
From Assumption \ref{ass3} it follows from \cite{sontag} that there exist a smooth function $S\left(Z\right): \mathbb{R}^{n+1}\to\mathbb{R}_+$ and class $\mathcal{K}_{\infty}$ functions $\alpha_1$, $\alpha_2$, $\alpha_3$, and $\alpha_4$, such that
\begin{eqnarray}
\!\!\!\!\!\!\alpha_1\left(\left|X\right|\right)\!\!\!&\leq&\!\!\! S\left(X\right)\leq\alpha_2\left(\left|X\right|\right)\label{issbg}\\
\!\!\!\!\!\!\frac{\partial S\left(X\right)}{\partial X}f\left(X,\kappa_1\left(X\right)+\omega_1,\ldots,\kappa_m\left(X\right)+\omega_m\right)\!\!\!&\leq&\!\!\! -\alpha_3\left(\left|X\right|\right)+\alpha_{4}\left(\sum_{i=1}^{i=m}|\omega_i|\right)\label{issg},
\end{eqnarray} 
for all $X\in\mathbb{R}^n$ and $\omega_1,\ldots,\omega_m\in\mathbb{R}$. With similar calculations as in \cite{krstic} (Theorem 5) we get that
\begin{eqnarray}
\frac{d\|w_i(t)\|_{c,\infty}}{dt}&\leq&-c\|w_i(t)\|_{c,\infty},\quad i=1,\ldots,m,\label{lypwnew}
\end{eqnarray}
along the solutions of (\ref{pde1tr}), (\ref{pde2tr}). Consider the Lyapunov functional
\begin{eqnarray}
V(t)&=&S(X(t))+\frac{2}{c}\int_0^{\sum_{i=1}^{i=m}\|w_i(t)\|_{c,\infty}}\frac{\alpha_{4}(r)}{r}dr.\label{lyapnew123}
\end{eqnarray}
Using Lemma \ref{leminvz} (relation (\ref{plant1tr})) and relations (\ref{issg}), (\ref{lypwnew}) we get along the solutions of (\ref{plant1tr})--(\ref{pde2tr}) that
\begin{eqnarray}
\dot{V}(t)&\leq&-\alpha_3\left(\left|X(t)\right|\right)+\alpha_{4}\left(\sum_{i=1}^{i=m}|w_i(0,t)|\right)-2\alpha_4\left(\sum_{i=1}^{i=m}\|w_i(t)\|_{c,\infty}\right).\label{lyapnew11}
\end{eqnarray}
Using the fact that $|w_i(0,t)|\leq\|w_i(t)\|_{c,\infty}$, $i=1,\ldots,m$, we get that
\begin{eqnarray}
\dot{V}(t)&\leq&-\alpha_3\left(\left|X(t)\right|\right)-\alpha_4\left(\sum_{i=1}^{i=m}\|w_i(t)\|_{c,\infty}\right).\label{lyapnew11n}
\end{eqnarray}
It follows with the help of (\ref{issbg}) that there exists a class $\mathcal{K}$ function $\gamma_1$ such that 
\begin{eqnarray}
\dot{V}(t)\leq-\gamma_1\left(V(t)\right),
\end{eqnarray}
and hence, with the comparison principle (see, for example, \cite{khalil}) one can conclude that there exists a class $\mathcal{KL}$ function $\beta_2$ such that
\begin{eqnarray}
V(t)\leq\beta_2\left(V(0),t\right).
\end{eqnarray}
Using (\ref{issbg}), the fact that $\|w_i(t)\|_{\infty}\leq\|w_i(t)\|_{c,\infty}\leq e^{cD_i}\|w_i(t)\|_{\infty}$, $i=1,\ldots,m$, and the properties of class $\mathcal{K}$ functions we get estimate (\ref{blem}).

\subsection*{Proof of Lemma \ref{lemmabdirect}}
\label{proof lemmabdirect}
We prove the lemma by induction. For clarity we present two initial steps. We prove first bound (\ref{pboundd}). Using the fact that $p_1$ satisfies (\ref{p1ode}) we get under Assumption \ref{ass1} that there exists a smooth function $R:\mathbb{R}^n\to\mathbb{R}_+$ and class $\mathcal{K}_{\infty}$ functions $\alpha_5$, $\alpha_6$, and $\alpha_7$ such that
\begin{eqnarray}
\alpha_5(|X|)&\leq& R(X)\leq\alpha_6(|X|)\label{for1}\\
\frac{\partial R(X)} {\partial X}f\left(X,\omega_1,\ldots,\omega_m\right)&\leq& R(X)+\alpha_7\left(\sum_{i=1}^{i=m}|\omega_i|\right),
\end{eqnarray}
for all $X\in\mathbb{R}^n$ and $\omega_i\in\mathbb{R}$, $i=1,\ldots,m$. Therefore,
\begin{eqnarray}
\frac{dR\left(p_1(x,t)\right)}{dx}&=&\frac{\partial R\left(p_1(x,t)\right)}{\partial p_1}f\left(p_1(x,t),u_1(x,t),\ldots,u_m(x,t)\right)\nonumber\\
&\leq& R\left(p_1(x,t)\right)+\alpha_7\left(\sum_{i=1}^{i=m}|u_i(x,t)|\right),\quad x\in[0,D_1].
\end{eqnarray}
Hence, using (\ref{p1b}) we get that
\begin{eqnarray}
R\left(p_1(x,t)\right)&\leq& e^{D_1}R\left(X(t)\right)+e^{D_1}\int_0^{D_1}\alpha_7\left(\sum_{i=1}^{i=m}|u_i(x,t)|\right)dx,
\end{eqnarray}
and hence,
\begin{eqnarray}
R\left(p_1(x,t)\right)&\leq& e^{D_1}R\left(X(t)\right)+D_1e^{D_1}\alpha_7\left(\sum_{i=1}^{i=m}\|u_i(t)\|_{\infty}\right).
\end{eqnarray}
With the help of (\ref{for1}) and the properties of class $\mathcal{K}$ functions we get estimate (\ref{pboundd}). We prove next (\ref{pboundd2}). Under Assumption \ref{ass2} one can conclude that there exists a smooth function $R_1:\mathbb{R}^n\to\mathbb{R}_+$ and class $\mathcal{K}_{\infty}$ functions $\alpha_8$, $\alpha_9$, and $\alpha_{10}$ such that
\begin{eqnarray}
\alpha_8(|X|)&\leq& R_1(X)\leq\alpha_9(|X|)\label{for1new}\\
\frac{\partial R_1(X)} {\partial X}g_1\left(X,\omega_2,\ldots,\omega_m\right)&\leq& R_1(X)+\alpha_{10}\left(\sum_{i=2}^{i=m}|\omega_i|\right),\label{for12}
\end{eqnarray}
for all $X\in\mathbb{R}^n$ and $\omega_i\in\mathbb{R}$, $i=2,\ldots,m$, where $g_1$ is defined in Assumption \ref{ass2}. Using the fact that $p_2$ satisfies (\ref{p2ode}) we get from (\ref{for12}) that
\begin{eqnarray}
\frac{dR_1\left(p_2(x,t)\right)}{dx}&=&\frac{\partial R_{1}\left(p_2(x,t)\right)}{\partial p}g_1\left(p_2(x,t),u_2(x,t),\ldots,u_m(x,t)\right)\nonumber\\
&\leq& R_{1}\left(p_2(x,t)\right)+\alpha_{10}\left(\sum_{i=2}^{i=m}|u_i(x,t)|\right),\quad x\in[D_1,D_2].
\end{eqnarray}
Thus, using (\ref{p2b}) we get that
\begin{eqnarray}
R_{1}\left(p_2(x,t)\right)&\leq& e^{D_2}R_{1}\left(p_1(D_1,t)\right)+e^{D_2}\int_0^{D_2}\alpha_{10}\left(\sum_{i=2}^{i=m}|u_i(x,t)|\right)dx,
\end{eqnarray}
and hence, from (\ref{for1new}) that
\begin{eqnarray}
R_{1}\left(p_2(x,t)\right)&\leq& e^{D_2}\alpha_9\left(\|p_1(t)\|_{\infty}\right)+D_2e^{D_2}\alpha_{10}\left(\sum_{i=2}^{i=m}\|u_i(t)\|_{\infty}\right).\label{new1}
\end{eqnarray}
Using (\ref{for1new}) and (\ref{pboundd}) we get from (\ref{new1}) estimate (\ref{pboundd2}). Assume now that for some $j$ it holds that
\begin{eqnarray}
\|p_j(t)\|_{\infty}\leq \rho_j\left(\Xi(t)\right).\label{inter}
\end{eqnarray}
Under Assumption \ref{ass2} there exist a smooth function $R_{j}:\mathbb{R}^n\to\mathbb{R}_+$ and class $\mathcal{K}_{\infty}$ functions $\alpha_{8+3(j-1)}$, $\alpha_{9+3(j-1)}$, and $\alpha_{10+3(j-1)}$ such that
\begin{eqnarray}
\alpha_{8+3(j-1)}(|X|)&\leq& R_j(X)\leq\alpha_{9+3(j-1)}(|X|)\label{for1newm}\\
\frac{\partial R_j(X)} {\partial X}g_j\left(X,\omega_{j+1},\ldots,\omega_m\right)&\leq& R_j(X)+\alpha_{10+3(j-1)}\left(\sum_{i=j+1}^{i=m}|\omega_i|\right),\label{for12m}
\end{eqnarray}
for all $X\in\mathbb{R}^n$ and $\omega_i\in\mathbb{R}$, $i=j+1,\ldots,m$. Using the fact that $p_{j+1}$ satisfies (\ref{pm1ode}) and definition $g_j\left(X,\omega_{j+1},\ldots,\omega_m\right)=f\left(X,\kappa_1\left(X\right),\ldots,\kappa_{j}\left(X\right),\omega_{j+1},\ldots,\omega_m\right)$ we get from (\ref{for12m}) that
\begin{eqnarray}
\frac{dR_j\left(p_{j+1}(x,t)\right)}{dx}&=&\frac{\partial R_{j}\left(p_{j+1}(x,t)\right)}{\partial p}g_j\left(p_{j+1}(x,t),u_{j+1}(x,t),\ldots,u_m(x,t)\right)\nonumber\\
&\leq& R_{j}\left(p_{j+1}(x,t)\right)+\alpha_{10+3(j-1)}\left(\sum_{i=j+1}^{i=m}|u_i(x,t)|\right),\quad x\in[D_j,D_{j+1}].
\end{eqnarray}
Therefore, employing (\ref{for1newm}) we get that
\begin{eqnarray}
R_{j}\left(p_{j+1}(x,t)\right)&\leq& e^{D_{j+1}}\alpha_{9+3(j-1)}\left(\|p_j(t)\|_{\infty}\right)+D_{j+1}e^{D_{j+1}}\alpha_{10+3(j-1)}\left(\sum_{i=j+1}^{i=m}\|u_i(t)\|_{\infty}\right).\label{new1j}
\end{eqnarray}
Using (\ref{for1newm}), (\ref{inter}), and the properties of class $\mathcal{K}$ functions the proof is completed. 

\subsection*{Proof of Lemma \ref{lemmabdirectinv}}
\label{proof lemmabdirectinv}
The proof of this lemma employs similar arguments to the proof of Lemma \ref{lemmabdirect} with the difference that one uses the ODEs in $x$ for $\pi_i$, $i=1,\ldots,m$ together with Assumption \ref{ass3}. We again prove this lemma by induction. We first prove (\ref{invbpi}). Using the fact that $\pi_1$ satisfies the initial value problem (\ref{pode1inv1}), (\ref{pode2inv}), we get under Assumption \ref{ass3} and the definition of input-to-state stability (see, for example, \cite{sontag}) that there exist a class $\mathcal{KL}$ function $\beta_3$ and a class $\mathcal{K}_{\infty}$ function $\gamma_2$ such that
\begin{eqnarray}
|\pi_1(x,t)|\leq \beta_3\left(|X(t)|,x\right)+\gamma_2\left(\sup_{0\leq y\leq x}\left(\sum_{i=1}^{i=m}|w_i(x,t)|\right)\right),\quad x\in[0,D_1],
\end{eqnarray}
and hence, we arrive at (\ref{invbpi}) with $\bar{\rho}_1(s)=\beta_3\left(s,0\right)+\gamma_2\left(s\right)$. Assume next that for some $j$ it holds that
\begin{eqnarray}
\|\pi_j(t)\|_{\infty}\leq \bar{\rho}_j\left(\bar{\Xi}(t)\right).\label{interinv}
\end{eqnarray}
Using the fact that $\pi_{j+1}$ satisfies the initial value problem (\ref{eqn1}), (\ref{eqn1b}) we get under Assumption \ref{ass3} that 
\begin{eqnarray}
|\pi_{j+1}(x,t)|\leq \beta_3\left(|\pi_j(D_j)|,x-D_j\right)+\gamma_2\left(\sup_{D_j\leq y\leq x}\left(\sum_{i=j+1}^{i=m}|w_i(x,t)|\right)\right),\quad x\in[D_j,D_{j+1}],
\end{eqnarray}
and hence, with (\ref{interinv}) we arrive at $\|\pi_{j+1}(t)\|_{\infty}\leq \bar{\rho}_{j+1}\left(\bar{\Xi}(t)\right)$, with $\bar{\rho}_{j+1}(s)=\beta_3\left(\bar{\rho}_j,0\right)+\gamma_2\left(s\right)$. 

\subsection*{Proof of Lemma \ref{lemmaeq}}
We prove first (\ref{b1eq}). Using the fact that $\kappa_i$, $i=1,\ldots,m$, are locally Lipschitz with $\kappa_i(0)=0$, $i=1,\ldots,m$, there exist class $\mathcal{K}_{\infty}$ functions ${\alpha}^*_i$, $i=1,\ldots,m$, such that
\begin{eqnarray}
\left|\kappa_i(X)\right|\leq{\alpha}^*_i\left(|X|\right),\quad i=1,\ldots,m,
\end{eqnarray}
for all $X\in\mathbb{R}^n$. Hence, using (\ref{eq2})--(\ref{bacmnew}) and relations (\ref{pboundd})--(\ref{pbounddm}) from Lemma \ref{lemmabdirect} we get estimate (\ref{b1eq}) with $\rho(s)=s+\sum_{i=1}^{i=m}{\alpha}^*_i\left(\sum_{j=1}^{j=m}\rho_j(s)\right)$. Similarly, using (\ref{eq2in})--(\ref{bacmnewinv}) we get estimate (\ref{b2eq}) by employing Lemma \ref{lemmabdirectinv} with $\bar{\rho}(s)=s+\sum_{i=1}^{i=m}{\alpha}^*_i\left(\sum_{j=1}^{j=m}\bar{\rho}_j(s)\right)$.



\end{document}